\documentclass[11pt,oneside,a4paper]{article}

\usepackage[english]{babel}
\usepackage[utf8]{inputenc} 
\usepackage{amsmath,amsfonts,amsthm} 
\usepackage{geometry}		
\usepackage{authblk}
\usepackage{graphicx}
\usepackage{fourier} 
\usepackage{times}
\usepackage{comment}
\usepackage{pgfplots}
\usepackage[T1]{fontenc}
\usepackage{mathrsfs}
\usepackage{latexsym}
\usepackage{amssymb}
\usepackage{float}
\usepackage{extarrows}
\usepackage{caption}
\usepackage{subcaption}
\usepackage{wrapfig}
\usepackage[hidelinks]{hyperref}
\usepackage{lipsum}
\usepackage{algpseudocode}
\usepackage{algorithm}
\usepackage{subcaption}
\usepackage{multicol}
\usepackage{array}

\newtheorem{proposition}{Proposition}
\newtheorem{remark}{Remark}
\newtheorem{theorem}{Theorem}
\newtheorem{corollary}{Corollary}
\newtheorem{example}{Example}

\numberwithin{equation}{section}
\DeclareMathAlphabet{\mathcal}{OMS}{cmsy}{m}{n}

\usepackage[activate={true,nocompatibility},
            final,tracking=true,
            kerning=true,
            spacing=true,
            factor=1100,
            stretch=10,
            shrink=10]{microtype}
\microtypecontext{spacing=nonfrench}

\title{Linearly implicit local and global energy-preserving methods for PDEs with a cubic Hamiltonian}
\author{Sølve Eidnes$^*$ and Lu Li%
\thanks{Department of Mathematical Sciences, NTNU, N--7491 Trondheim, Norway.\\Sølve Eidnes: \texttt{solve.eidnes@ntnu.no}; Lu Li (corresponding author): \texttt{lu.li@ntnu.no}.}}

\begin{document}
\maketitle
\begin{abstract}
We present linearly implicit methods that preserve discrete approximations to local and global energy conservation laws for multi-symplectic PDEs with cubic invariants. The methods are tested on the one-dimensional Korteweg--de Vries equation and the two-dimensional Zakharov--Kuznetsov equation; the numerical simulations confirm the conservative properties of the methods, and demonstrate their good stability properties and superior running speed when compared to fully implicit schemes.

\vspace{4pt}
\textbf{Keywords}: Structure-preserving methods, multi-symplectic PDEs, Kahan's method.

\vspace{2pt}
\textbf{Classification}: 37K05, 65M06, 65P10
\end{abstract}

\section{Introduction} 
In recent years, much attention has been given to the design and analysis of numerical methods for differential equations that can capture geometric properties of the exact flow. The increased interest in this subject can mainly be attributed to the superior qualitative behaviour over long time integration of such structure-preserving methods, see \cite{hairer2006geometric, furihata2011discrete, christiansen2011topics}. 
A popular class of structure-preserving methods are energy-preserving methods. Energy preservation has a far-reaching importance throughout the physical sciences \cite{Feynmanbook, bridges1997multi}. In particular, it has been found to be crucial in the proof of stability for several numerical methods, see e.g. \cite{li1995finite}.

Energy-preserving methods are well studied for finite-dimensional Hamiltonian systems \cite{labudde1975energy,mclachlan1999geometric,brugnano10hbv,celledoni2012preserving}. It is also highly conceivable that the ideas behind the finite-dimensional setting can be extended to the infinite-dimensional Hamiltonian systems or Hamiltonian partial differential equations (PDEs) \cite{bridges2001multi}. There are two popular ways to construct energy-preserving methods for Hamiltonian PDEs. One approach is to semi-discretize the PDE in space so that one obtains a system of Hamiltonian ordinary differential equations (ODEs), and then apply an energy-preserving method to this semi-discrete system, see for example \cite{celledoni2012preserving}. In this way, it is straightforward to generalise the energy-preserving methods for finite-dimensional Hamiltonian systems to Hamiltonian PDEs. However, such methods  conserve only a global energy that relies on a proper boundary condition, such as a periodic boundary condition. If this is not present, the energy-preserving property will be destroyed. The other approach is based on a reformulation of the Hamiltonian PDE into a multi-symplectic form, which provides the PDE with three local conservation laws: the multi-symplectic conservation law, the energy conservation law and the momentum conservation law \cite{bridges1997multi,bridges1997geometric,leimkuhler2004simulating}. Then one may consider methods that preserve the local conservation laws, see for example multi-symplectic integrators \cite{Sunmultisym} and integrators which preserve the energy conservation law or the momentum conservation law \cite{EPmultisym, Hydonlocalconserv}. Although multi-symplectic integrators, like symplectic integrators, are proven to give good numerical approximations of the solution, such methods can only preserve quadratic invariants \cite{MR2221062, EPmultisym}. 
The focus of this paper is on integrators that preserve the local energy conservation law \cite{wang2008local}. These locally defined properties are not dependent on the choice of boundary conditions, giving the methods that preserve local energy an advantage over methods that preserve a global energy, especially since local conservation laws will always lead to global conservation laws whenever periodic boundary conditions are considered. 
The concept of a multi-symplectic structure for PDEs was introduced by Bridges in \cite{bridges1997multi,bridges1997geometric}, see also \cite{marsden1998multisymplectic} for a framework based on a Lagrangian formulation of the Cartan form. 
Local energy-preserving methods were first studied in \cite{reich2000multi}, and have garnered much interest recently, see for example \cite{wang2008local, gong2014some,li2015general}.

Most of the local energy-preserving methods proposed so far are fully implicit methods, for which a non-linear system must be solved at each time step.  This is normally done by using an iterative solver where a linear system is solved at each iteration, which can lead to computationally expensive procedures, especially since the number of iterations needed in general increases with the size of the system. A fully explicit method on the other hand, may over-simplify the problem and often has inferior stability properties, so that a strong restriction on the grid ratio is needed. A good alternative may therefore be to develop linearly implicit schemes, where the solution at the next time step is found by solving only one linear system. 

One example of linearly implicit methods for Hamiltonian ODEs is Kahan's method,  which was designed for solving quadratic ODEs \cite{kahan1993unconventional} and whose geometric properties have been studied in a series of papers by Celledoni et al.\ \cite{celledoni2012geometric, celledoni2012integrability, celledoni2018geometric}. For Hamiltonian PDEs, Matsuo and Furihata proposed the idea of using multiple points to discretize the variational derivative and thus design linearly implicit energy-preserving schemes \cite{matsuo2001dissipative}. Dahlby and Owren generalised this concept and developed a framework for deriving linearly implicit energy-preserving multi-step methods for Hamiltonian PDEs with polynomial invariants \cite{dahlby2011general}. A comparison of this approach and Kahan's method applied to PDEs is given in \cite{eidnes2019linearly}. Recently, more work has been put into developing linearly implicit energy-preserving schemes for Hamiltonian PDEs, e.g.\ the partitioned averaged vector field (PAVF) method \cite{cai2018partitioned} and schemes based on the invariant energy quadratization (IEQ) approach \cite{yang2017numerical} or the multiple scalar auxiliary variables (MSAV) approach \cite{jiang2019linearly}.
However, little attention has been given to linearly implicit local energy-preserving methods. To the best of the authors' knowledge, the only existing method is one based on the IEQ approach, specific for the sine-Gordon equation \cite{jiang2018linear}. In this paper, we use Kahan's method to construct a linearly implicit method that preserves a discrete approximation to the local energy for multi-symplectic PDEs with a cubic energy function. This class is extensive, and includes among other PDEs the Korteweg--de Vries (KdV) equation, the Benjamin--Bona--Mahony (BBM) equation \cite{li13new}, Boussinesq-type systems \cite{Boussinesq}, and the Camassa--Holm equation \cite{Cohnmultisym}.

The rest of this paper is organized as follows. First, we give an overview of Kahan's method and formulate it by using a polarised energy function. A brief introduction to multi-symplectic PDEs and their conservation laws are presented in Section \ref {multi-symplectic PDES properties}. In Section \ref{New schemes}, new linearly implicit local and global energy-preserving schemes are presented. Numerical examples for the KdV and Zakharov--Kuznetsov equations are given in Section \ref {Numerical example}, before we end the paper with some concluding remarks.

\section{Kahan's method}\label{Kahan}
Consider an ODE system
\begin{equation}
\dot{y} = f(y) = \hat{Q}(y)+\hat{B}y+\hat{c}, \quad y \in \mathbb{R}^M,
\label{original quadratic form}
\end{equation} 
where $\hat{Q}(y)$ is an $\mathbb{R}^M$ valued quadratic form, $\hat{B}\in \mathbb{R}^{M\times M}$ is a symmetric constant matrix, and $\hat{c} \in \mathbb{R}^{M}$ is a constant vector. Kahan's method is then given by 
\begin{align*}
\frac {y^{n+1}-y^n}{\Delta t} =\bar{Q}(y^n,y^{n+1})+\hat{B}\frac {y^{n}+y^{n+1}}{2} +\hat{c},
\end{align*}
where 
\begin{align*}
\bar{Q}(y^{n},y^{n+1}) = \frac {1}{2}\big(\hat{Q}(y^{n}+y^{n+1})-\hat{Q}(y^{n})-\hat{Q}(y^{n+1})\big)
\end{align*}
is the symmetric bilinear form obtained by polarisation of the quadratic form $\hat{Q}$ \cite{celledoni2012geometric}. Polarisation, which maps a homogeneous polynomial function to a symmetric multi-linear form in
more variables, was used to generalise Kahan's method to higher degree polynomial vector fields in \cite{celledoni2015discretization}.

Suppose we restrict the problem \eqref{original quadratic form} to be a Hamiltonian system on a Poisson vector space with a constant Poisson structure:
\begin{equation}\label{original eqode-quadratic}
\dot{y} = A\nabla H(y), 
\end{equation}
where $A$ is a constant skew-symmetric matrix, and $H:\mathbb{R}^M \rightarrow \mathbb{R}$ is a cubic polynomial function. We first consider the Hamiltonian $H$ to be homogeneous.
Then, following the result in Proposition $2.1$ of \cite{celledoni2015discretization}, Kahan's method can be reformulated as
\begin{align}\label{Kahan-map with H tensor}
 \frac{y^{n+1}-y^n}{\Delta t} = 3 A\bar{H}(y^n,y^{n+1},\cdot),
\end{align}
where $\bar{H}(\cdot,\cdot,\cdot):\mathbb{R}^M\times\mathbb{R}^M\times\mathbb{R}^M \rightarrow \mathbb{R}$ is a symmetric $3$-tensor satisfying $\bar{H}(x,x,x) = H(x)$. Consider the $3$-tensor $\bar{H}(x,y,z)=x^TQ(y)z$, where $ Q(y)=\frac{1}{6}\nabla^2 H(y)$, with $\nabla^2 H$ being the Hessian of $H$; then we can rewrite Kahan's method \eqref{Kahan-map with H tensor} as
\begin{align}\label{Kahan-map-tensor-partial}
 \frac{y^{n+1}-y^n}{\Delta t} = 3A\frac {\partial \bar{H}}{\partial x}\bigg\lvert_{(y^n,y^{n+1})},
\end{align}
where $\frac {\partial \bar{H}}{\partial x}$ denotes the partial derivative with respect to the first argument of $\bar{H}$.

Consider then the cases where the Hamiltonian in problem \eqref{original eqode-quadratic} is non-homogen\-eous, i.e.\ of the general form
\begin{align}\label{energy nonhomogeneous}
 H(y)=y^TQ(y)y+y^TBy+c^Ty+d,
\end{align}
where $Q(y)$ is the linear part of $\nabla^2 H(y)$ and thus a symmetric matrix whose elements are homogeneous linear polynomials, $B$ is the constant part of $\nabla^2 H(y)$ and thus a symmetric constant matrix, $c$ is a constant vector and $d$ is a constant scalar. We follow the technique in \cite{celledoni2012geometric}, adding one variable to $y = (y_1,\dots,y_M)^T$ to get $\tilde{y}=(y_0,y_1,\dots,y_M)^T$, extending $A$ to $\tilde{A}$ by adding a zero initial row and a zero initial column, considering a homogeneous function $\tilde{H}(\tilde{y})$ based on the non-homogeneous Hamiltonian $H(y)$ such that $\tilde{H}(\tilde{y})\lvert_{y_0=1}=H(y)$, and finally
solving instead of \eqref{original eqode-quadratic} the equivalent, homogeneous cubic Hamiltonian problem
\begin{equation*}
\dot{\tilde{y}}=\tilde{A}\nabla \tilde{H}(\tilde{y})
\end{equation*}
with $y_0=1$. In this way we can still get the reformulation of Kahan's method as \eqref{Kahan-map-tensor-partial} with
\begin{align}\label{eq:polarisedH}
\bar{H}(x,y,z)  = x^TQ(y)z+\frac {1}{3} (x^T B y + y^T B z + z^T B x)+ \frac{1}{3} c^T (x+y+z) + d.
\end{align} 
\begin{remark}
The $\mathbb{R}$-valued function $\bar{H}(x,y,z)$ in \eqref{eq:polarisedH} has the following properties:
\begin{enumerate}  
\item $\bar{H}(x,y,z)$ is symmetric\footnote{Denote the elements in $Q(y)$ by $q_{ij}y=\sum_{k}q_{ij}^ky_k$, where $q_{ij}^k$, $i,j,k=1,\cdots,M$, are scalars and $y_k$ is the $k$th element of $y$. We have that $q_{ij}^k$ satisfies $q_{ij}^k=q_{ki}^j=q_{jk}^i$ since $q_{ij}^k=\frac{1}{6}\frac{\partial^3 \bar{H}}{\partial y_i\partial y_j\partial y_k}$, which is unchanged under any permutation of $i,j, k$. This provides the symmetry of  $\bar{H}(x,y,z)$.} w.r.t. $x$, $y$ and $z$, 
\item $\bar{H}(x,x,x)=H(x)$,
\item $\frac{\partial \bar{H}(x,y,z)}{\partial x}=Q(y)z+\frac {B(y+z)}{3}+\frac {c}{3}$ is symmetric w.r.t. $y$ and $z$.
\end{enumerate}\label {Kahan's scheme with property of the polarized energy}
\end{remark}
In this paper, we will use the form of Kahan's method in \eqref{Kahan-map-tensor-partial} to prove the energy preservation of the proposed methods.

\section{Conservation laws for multi-symplectic PDEs}\label{multi-symplectic PDES properties}
Many PDEs, including all one-dimensional Hamiltonian PDEs, can be written on the multi-symplectic form
\begin{equation}\label{Multisym_PDEs}
K z_t+ L z_x=\nabla S(z), \quad z \in \mathbb{R}^l,\quad (x,t)\in \mathbb{R} \times \mathbb{R},
      \end{equation}
      where $K$, $L \in \mathbb{R}^{l\times l}$ are two constant skew-symmetric matrices and $S: \mathbb{R}^l \mapsto \mathbb{R}$ is a scalar-valued function. Following the results about multi-symplectic structure in \cite{bridges1997multi}, it can be shown that multi-symplectic PDEs satisfy the following local conservation laws \cite{moore2003multi}: the multi-symplectic conservation law
     \begin{equation*}\label{MultiPDEs_multisimplectic conservation law}
\partial_t \omega+\partial_x\kappa=0, \quad\omega=dz\wedge K_+dz,\quad\kappa=dz\wedge L_+dz,
      \end{equation*}    
the local energy conservation law (LECL) 
   \begin{equation}\label{MultiPDEs_energy conservation law}
E_t +F_x=0, \quad E=S(z)+z_x^TL_+z,\quad F=-z_t^TL_+z,
      \end{equation}    
and the local momentum conservation law (LMCL)
  \begin{equation*}\label{MultiPDEs_momentum conservation law}
I_t +G_x=0, \quad G=S(z)+z_t^T K_+z,\quad I=-z_x^T K_+z,
      \end{equation*}    
 where $K_+$ and $L_+$ satisfy
 $$K=K_+-K_+^T,   \quad    L=L_+-L_+^T.$$ 
Decomposition of the matrices is done to make deduction of the conservation laws for energy and momentum more efficient \cite[Section~12.3.1]{leimkuhler2004simulating}.

The multi-symplectic form \eqref{Multisym_PDEs} can also be generalised to problems in higher dimensional spaces. Consider $d$ spatial dimensions; based on the work by Bridges \cite{bridges1997multi}, a multi-symplectic PDE can then be written as
\begin{equation}\label{Multisym_PDEs_multidimension}
K z_t+\sum_{\alpha=1}^{d}L^\alpha z_{x_\alpha}=\nabla S(z), \quad z \in \mathbb{R}^l,\quad (x,t)\in \mathbb{R}^d\times\mathbb{R},
      \end{equation}
where $K$, $L^\alpha \in \mathbb{R}^{l\times l}$ $(\alpha=1,\ldots,d)$ are constant skew-symmetric matrices and  $S:\mathbb{R}^l\rightarrow\mathbb{R}$ is a smooth functional. Equation \eqref{Multisym_PDEs_multidimension} has the following local energy conservation law:
   \begin{equation}\label{MultiPDEs energy conservation law}
E_t +\sum_{\alpha=1}^{d}F^{\alpha}_{x_{\alpha}}=0,
      \end{equation}    
where $E(z)=S(z)+\sum_{\alpha=1}^{d}z_{\alpha}^T L^{\alpha}_+z$, $F^{\alpha}=-z_t^TL^{\alpha}_+z$, and $L^{\alpha}_+$ are splittings of $L^{\alpha}$ satisfying $L^{\alpha}=L^{\alpha}_+-({L^{\alpha}_+})^T$.

Say we have \eqref{Multisym_PDEs_multidimension} defined on the spatial domain $\Omega\in\mathbb{R}^d$ with periodic boundary conditions. Integrating over the domain $\Omega$ on both sides of the equation \eqref{MultiPDEs energy conservation law} and using the periodic boundary condition then leads to the global energy conservation law for the multi-symplectic PDEs,
  \begin{equation}\label{global_EC}
\frac{d}{dt}\mathcal{E}(z) =0,
  \end{equation}
  where $\mathcal{E}(z)=\int_\Omega E(z)d\Omega$.

\begin{example}\label{ex:KdV}
\textbf{Korteweg--de Vries equation.} Consider the KdV equation for modeling shallow water waves,
\begin{equation}\label{eq_KdV}
u_t + \eta u u_x+\gamma^2 u_{xxx}  = 0,
\end{equation}
where $\eta, \gamma \in \mathbb{R}$. Introducing the potential $\phi_x=u$, momenta $v=\gamma u_x$ and the variable $w=\gamma v_x \phi_t+\frac{\gamma^2u^2}{2}$ by the covariant Legendre transform from the Lagrangian, we obtain
\begin{equation}\label{Kdv_multisym_element}
\begin{split}
\frac{1}{2}u_t+w_x&=0,\\
-\frac{1}{2}\phi_t-\gamma v_x&=-w+\frac{\eta}{2}u^2,\\
\gamma u_x&=v,\\
-\phi_x&=-u,
\end{split}
\end{equation}
from which we find the multi-symplectic formulation \eqref{Multisym_PDEs} for the KdV equation with $z = (\phi,u,v,w)^T$, the Hamiltonian
$S(z)=\frac{v^2}{2}-uw+\frac{\eta u^3}{6}$, and
\begin{equation*}
K=\begin{bmatrix}
    0       &\frac{1}{2} & 0 & 0 \\
    -\frac{1}{2}       & 0 & 0 & 0 \\
      0      & 0 & 0 & 0 \\
      0      & 0 & 0 & 0  \\
\end{bmatrix},\qquad
L=
\begin{bmatrix}
    0       &0 & 0 &1  \\
   0       & 0 & -\gamma & 0  \\
      0      & \gamma & 0 & 0  \\
      -1      & 0 & 0 & 0 \\
\end{bmatrix}.
\end{equation*}
As for the conservation laws, there are many choices of $K_+$ and $L_+$, for example $K_+=\frac{K}{2}, L_+=\frac{L}{2}$, or $K_+$ and $L_+$ being the upper triangular parts of $K$ and $L$, respectively.
\end{example}

\begin{example}\label{ex:ZK}
\textbf{Zakharov--Kuznetsov equation.}
Zakharov and Kuznetsov introduced in \cite{zakharov1974threedimensional} a (2+1)-dimensional generalisation of the KdV equation which includes weak transverse variation,
\begin{equation}
u_t + u u_x + u_{xxx} + u_{x y y}=0.
\label{eq:Zakharov-Kuznetsov}
\end{equation}
A multi-symplectification of this leads to a system (\ref{Multisym_PDEs_multidimension}) for two spatial dimensions,
\begin{equation}\label{Multisym_PDEs_3D}
K z_t + L^1 z_x + L^2 z_y = \nabla S(z), \quad z \in \mathbb{R}^6, \quad (x,y,t)\in \mathbb{R}^2 \times \mathbb{R}.
\end{equation}
Following \cite{bridges2001multi}, we have that \eqref{eq:Zakharov-Kuznetsov} is equivalent to a system of first-order PDEs,
\begin{equation}\label{eq:ZKsystem}
\begin{split}
\phi_x &= u,\\
\frac{1}{2}\phi_t + v_x + w_y &= p - \frac{1}{2}u^2,\\
w_x - v_y &= 0,\\
-\frac{1}{2}u_t - p_x &= 0,\\
-u_x + q_y &= -v,\\
-q_x - u_y &= -w,
\end{split}
\end{equation}
which is \eqref{Multisym_PDEs_3D} with $z = (p,u,q,\phi,v,w)^T$, the Hamiltonian $S(z) = u p - \frac{1}{2}(v^2+w^2) - \frac{1}{6}u^3$, and the skew-symmetric matrices $K, L^1, L^2$ whose only non-zero elements are
\begin{align*}
&k_{2,4} = \frac{1}{2}, \quad k_{4,2} = -\frac{1}{2},\\
&l^1_{1,4} = l^1_{2,5} = l^1_{3,6} = 1, \quad l^1_{4,1} = l^1_{5,2} = l^1_{6,2} = -1,\\
&l^2_{2,6} = l^2_{5,3} = 1, \quad l^2_{6,2} = l^2_{3,5} = -1.
\end{align*}
\end{example}
   
\section{New linearly implicit energy-preserving schemes}\label{New schemes}
In \cite{gong2014some}, Gong, Cai and Wang present a scheme that preserves the local energy conservation law \eqref{MultiPDEs_energy conservation law} of a one-dimensional multi-symplectic PDE, obtained by applying the midpoint rule in space and the averaged vector field (AVF) method in time. 
They also present schemes that preserve the global energy, but not \eqref{MultiPDEs_energy conservation law}, obtained by considering spatial discretizations that preserve the skew-symmetric property of the difference operator $\partial_x$. We build on their work by considering Kahan's method for the discretization in time, ensuring linearly implicit schemes and also energy preservation.

To introduce our new schemes, we begin with some basic difference operators:
\begin{align*}
\delta_{t} v_j^n &:= \frac{v_j^{n+1} -v_j^n}{\Delta t}, & \delta_{x} v_j^n &:= \frac{v_{j+1}^n -v_{j}^n}{\Delta x}\\
\mu_{t} v_j^n &:= \frac{v_j^{n+1} +v_j^n}{2},& \mu_{x}v_j^n &:= \frac{v_{j+1}^{n} +v_j^n}{2}.
\end{align*}
The operators satisfy the following properties \cite{wang2008local}:
\begin{enumerate} 
		\item All the operators commute with each other, e.g.
		
		$\delta_{t}\delta_{x}v_j^n=\delta_{x}\delta_{t}v_j^n,\quad \delta_{t}\mu_{x}v_j^n=\mu_{x}\delta_{t}v_j^n,\quad \mu_{t}\delta_{x}v_j^n=\delta_{x}\mu_{t}v_j^n$.
		\item They satisfy the discrete Leibniz rule
		$$\delta_{t}(uv)_j^n=(\varepsilon u_j^{n+1}+(1-\varepsilon )u_j^n)\delta_{t}v_j^n+\delta_{t}u_j^n((1-\varepsilon) v_j^{n+1}+ \varepsilon v_j^{n}),\quad 0\le\varepsilon\le 1.$$
		Specifically,
\begin{align*}
\delta_{t}(uv)_j^n&=u_j^{n}\delta_{t}v_j^n+\delta_{t}u_j^nv_j^{n+1},\qquad\text{for}\quad \varepsilon=0,\\
\delta_{t}(uv)_j^n&=\mu_{t}u_j^{n}\delta_{t}v_j^n+\delta_{t}u_j^n\mu_{t}v_j^{n},\quad\text{for}\quad \varepsilon=\frac{1}{2},\\
\delta_{t}(uv)_j^n&=u_j^{n+1}\delta_{t}v_j^n+\delta_{t}u_j^nv_j^{n},\qquad\text{for}\quad \varepsilon=1.
      \end{align*}
	\end{enumerate}
One can obtain a series of similar commutative equations and discrete Leibniz rules that are not presented here, but which are also crucial in the proofs of the preservation properties of the schemes to be introduced in the remainder of this section.

\subsection{A local energy-preserving scheme for multi-symplectic PDEs}\label{Local energy-preserving algorithm for multi-symplectic PDEs}
In this section, we apply the midpoint rule in space and Kahan's method in time to construct a local energy-preserving method for multi-symplectic PDEs. Introducing the concept by first considering the one-dimensional system \eqref{Multisym_PDEs}, we apply the midpoint rule in space to get
  \begin{equation*}
K\partial_t \mu_x z_j+L\delta_xz_j=\nabla S(\mu_x z_j), \quad j=0,\ldots,M-1.
      \end{equation*}  
Then applying Kahan's method gives us the linearly implicit local energy-preserving (LILEP) scheme
  \begin{equation}\label{eq:LILEP}
K\delta_{t}\mu_{x}z_j^n+L\delta_x \mu_{t}z_j^n= 3\frac {\partial \bar{S}}{\partial x}\bigg\rvert_{(\mu_x z_j^n,\mu_x z_j^{n+1})}.
      \end{equation} 
Here we consider $S$ of the form $S(y)=y^TQ(y)y+y^TBy+c^Ty+d$, as in \eqref{energy nonhomogeneous}, and accordingly $\bar{S}(x,y,z)$ of the form \eqref{eq:polarisedH}.
      \begin{theorem}\label{proof of ECL for LECL scheme}
 The scheme \eqref{eq:LILEP} satisfies the discrete local energy conservation law
  \begin{equation}\label{eq:discreteLECL}
\delta_{t} (\bar{E}_L)_{j}^n +\delta_{x}(\bar{F}_L)_j^{n}=0,
      \end{equation}   
      where
  \begin{align}
  \begin{split}
 (\bar{E}_L)_{j}^n&=\bar{S}(\mu_x z_j^n, \mu_x z_j^n, \mu_x z_j^{n+1})\\
 &+\frac{1}{3}(\delta_{x}z_j^n)^TL_+\mu_x z_j^n+\frac{1}{3}(\delta_{x}z_j^n)^TL_+\mu_x z_j^{n+1}+\frac{1}{3}(\delta_{x}z_j^{n+1})^TL_+\mu_x z_j^{n},\end{split} \label{eq:localenergy} \\ 
 (\bar{F}_L)_j^{n}&=-\frac{1}{3}(\delta_{t}z_j^n)^TL_+\mu_t z_j^{n} - \frac{1}{3}(\delta_{t}z_j^n)^TL_+\mu_t z_j^{n+1} - \frac{1}{3}(\delta_{t}z_j^{n+1})^TL_+\mu_{t}z_j^n. \nonumber
      \end{align}         
  \end{theorem}   
   
\begin{proof}  
Taking the inner product with $\frac {1}{3}\delta_{t}\mu_x z_j^n$ on both sides of \eqref{eq:LILEP} and using the skew-symmetry of matrix $K$, we have
\begin{equation}\label{part1}
\frac {1}{3}(\delta_{t}\mu_x z_j^n)^T L \delta_x \mu_t z_j^n=(\delta_{t} \mu_x z_j^n)^T\frac {\partial\bar{S}}{\partial x}\bigg\rvert_{(\mu_x z_j^n,\mu_x z_j^{n+1})}.
\end{equation}
Taking the inner product with $\frac {1}{3}\delta_{t}\mu_xz_j^{n+1}$ on both sides of \eqref{eq:LILEP}, we get
 \begin{equation}\label{part2}
\frac {1}{3}(\delta_{t}\mu_x z_j^{n+1})^T K\delta_{t}\mu_{x}z_j^n+\frac {1}{3}(\delta_{t}\mu_xz_j^{n+1})^TL\delta_x\mu_{t}z_j^n=(\delta_{t}\mu_xz_j^{n+1})^T\frac {\partial\bar{S}}{\partial x}\bigg\rvert_{(\mu_xz_j^n,\mu_xz_j^{n+1})}.
\end{equation}  
Taking the inner product with $\frac {1}{3}\delta_{t}\mu_x z_j^{n}$ on both sides of the scheme \eqref{eq:LILEP} for the next time step, we get
 \begin{equation}\label{part3}
\frac {1}{3}(\delta_{t}\mu_x z_j^{n})^T K\delta_{t}\mu_{x}z_j^{n+1}+\frac {1}{3}(\delta_{t}\mu_xz_j^{n})^TL\delta_x \mu_{t}z_j^{n+1}=(\delta_{t}\mu_xz_j^{n})^T\frac {\partial\bar{S}}{\partial x}\bigg\rvert_{(\mu_xz_j^{n+1},\mu_xz_j^{n+2})}.
\end{equation}  
Adding equations \eqref{part1}, \eqref{part2} and \eqref{part3} and using the skew-symmetry of matrix $K$, we obtain
 \begin{equation}\label{3part_sum}
 \begin{split}
\frac {1}{3}\Big(&(\delta_{t}\mu_xz_j^{n})^TL\delta_x \mu_{t}z_j^n+(\delta_{t}\mu_x z_j^{n+1})^TL\delta_x \mu_{t}z_j^n+(\delta_{t}\mu_x z_j^{n})^TL\delta_x \mu_{t}z_j^{n+1}\Big)\\
=& \, (\delta_{t}\mu_xz_j^{n})^T\frac {\partial\bar{S}}{\partial x}\mid_{(\mu_x z_j^{n},\mu_x z_j^{n+1})}+(\delta_{t}\mu_x z_j^{n+1})^T\frac {\partial\bar{S}}{\partial x}\bigg\rvert_{(\mu_x z_j^{n},\mu_x z_j^{n+1})} \\
& \, + (\delta_{t}\mu_x z_j^{n})^T\frac {\partial\bar{S}}{\partial x}\bigg\rvert_{(\mu_x z_j^{n+1},\mu_x z_j^{n+2})},\\
=& \, \frac {1}{\Delta t}\big(\bar{S}(\mu_xz_j^{n+1},\mu_xz_j^{n+1},\mu_xz_j^{n+2})-\bar{S}(\mu_xz_j^{n},\mu_xz_j^{n},\mu_xz_j^{n+1})\big),\\
=& \, \delta_t \bar{S}(\mu_x z_j^{n},\mu_x z_j^{n},\mu_x z_j^{n+1}).
\end{split}
\end{equation}  
On the other hand, using the aforementioned commutative laws and discrete Leibniz rules for the operators, we can deduce
\begin{equation} \label{operator cancle}
\begin{split}
 \delta_{t}((\delta_xz_j^n)^TL_+\mu_x z_j^n)&=(\delta_{t}\delta_xz_j^n)^TL_+\mu_{t}\mu_x z_j^n+(\delta_x \mu_t z_j^n)^TL_+\delta_{t}\mu_x z_j^n,\\
\delta_{x}((\delta_tz_j^n)^TL_+\mu_t z_j^n)&=(\delta_{t}\delta_xz_j^n)^TL_+\mu_{t}\mu_xz_j^n+(\delta_t \mu_x z_j^n)^TL_+\delta_{x}\mu_t z_j^n,
\\
\delta_{t}((\delta_xz_j^{n+1})^TL_+\mu_x z_j^n)&=(\delta_{t}\delta_x z_j^{n+1})^TL_+\mu_{t}\mu_xz_j^{n}+(\delta_x\mu_tz_j^{n+1})^TL_+\delta_{t}\mu_xz_j^n,\\
\delta_{x}((\delta_tz_j^{n+1})^TL_+\mu_tz_j^{n})&=(\delta_{t}\delta_xz_j^{n+1})^TL_+\mu_{t}\mu_xz_j^{n}+(\delta_t\mu_xz_j^{n+1})^TL_+\delta_{x}\mu_tz_j^{n},
\\
\delta_{t}((\delta_xz_j^n)^TL_+\mu_xz_j^{n+1})&=(\delta_{t}\delta_x z_j^n)^TL_+\mu_{t}\mu_xz_j^{n+1}+(\delta_x\mu_tz_j^n)^TL_+\delta_{t}\mu_xz_j^{n+1},\\
\delta_{x}((\delta_tz_j^n)^TL_+\mu_tz_j^{n+1})&=(\delta_{t}\delta_xz_j^n)^TL_+\mu_{t}\mu_xz_j^{n+1}+(\delta_t\mu_xz_j^n)^TL_+\delta_{x}\mu_tz_j^{n+1}.
\end{split}
\end{equation}
Using the above relations \eqref{operator cancle}, the fact that $L=L_+-L_+^T$ and the result \eqref{3part_sum}, we obtain
\begin{align*}
\delta_t E_{j }^n+\delta_x F^{n }_j=& \, \delta_t \bar{S}(\mu_xz_j^{n},\mu_xz_j^{n},\mu_xz_j^{n+1})\\
& \, +\frac{1}{3}\big( \delta_{t}((\delta_xz_j^n)^TL_+\mu_xz_j^n)+\delta_{t}((\delta_xz_j^n)^TL_+\mu_xz_j^{n+1})\\
& \, +\delta_{t}((\delta_xz_j^{n+1})^TL_+\mu_xz_j^n)\big) - \frac{1}{3}\big(\delta_{x}((\delta_tz_j^n)^TL_+\mu_tz_j^n)\\
& \, +\delta_{x}((\delta_tz_j^n)^TL_+\mu_tz_j^{n+1})+\delta_{x}((\delta_tz_j^{n+1})^TL_+\mu_tz_j^{n})\big)\\
=& \, \delta_t \bar{S}(\mu_xz_j^{n},\mu_xz_j^{n},\mu_xz_j^{n+1}) - \frac {1}{3}\big((\delta_{t}\mu_xz_j^{n})^TL\delta_x\mu_{t}z_j^n\\
& \, +(\delta_{t}\mu_xz_j^{n+1})^TL\delta_x\mu_{t}z_j^n+(\delta_{t}\mu_xz_j^{n})^TL\delta_x\mu_{t}z_j^{n+1}\big)\\
=& \, 0.
\end{align*}
\end{proof}       

 \begin{corollary}\label{GEC for LECL scheme}
 For periodic boundary conditions $z(x+P,t)=z(x,t)$, the sche\-me \eqref{eq:LILEP} satisfies the discrete global energy conservation law
  \begin{equation}\label{discrete_ECL}
\bar{\mathcal{E}}_L^{n+1} =\bar{\mathcal{E}}_L^{n}, \quad \bar{\mathcal{E}}_L^{n}:=\Delta x\sum_{j=0}^{M-1} (\bar{E}_L)_{j}^n,
  \end{equation}
  where $\Delta x = P/M$ and $(\bar{E}_L)_{j}^n$ is given by \eqref{eq:localenergy}. 
  \end{corollary}
  \begin{proof}
  With periodic boundary conditions, we get
  $
  \sum_{j=0}^{M-1} \delta_{x}(\bar{F}_L)_j^{n} = 0,
  $
  and thus \eqref{discrete_ECL} follows from \eqref{eq:discreteLECL}.
  \end{proof}
  
The polarised global energy $\bar{\mathcal{E}}_L^{n}$ may be considered as a function of the solution in time step $n$ only, similarly to the modified Hamiltonian defined in Proposition 3 of \cite{celledoni2012geometric}.
\begin{proposition}\label{prop:modE}
With the solution $z^{n+1}$ found from $z^n$ by \eqref{eq:LILEP}, the discrete global energy $\bar{\mathcal{E}}_L^{n}$ of \eqref{discrete_ECL} satisfies
\begin{equation}
\bar{\mathcal{E}}_L^{n} = \mathcal{E}_L^{n} + \Delta x \sum_{j=0}^{M-1} \frac{1}{3}(\nabla E_L(z_j^n))^T (z_j^{n+1}-z_j^n),
\label{eq:modeneergy}
\end{equation}
where
\begin{align}
\mathcal{E}_L^{n} := \Delta x\sum_{j=0}^{M-1} E_L(z_{j}^n), \quad E_L(z_{j}^n) := S(\mu_x z_j^n) + (\delta_x z_j^n)^T L_+ \mu_x z_j^n,
\label{eq:fullenergy}
\end{align}
while $z_j^{n+1}-z_j^n$ satisfies
\begin{align}\label{eq:newfrm}
R_L(z_j^n)(z_j^{n+1}-z_j^n)=\Delta t g_L(z_j^n),
\end{align}
with $g_L(z_j^n)=\nabla S(\mu_x z_j^n)-L\delta_xz_j^n$ and $R_L(z_j^n)=K\mu_x-\frac{\Delta t}{2}\nabla g_L(z_j^n)$.
\end{proposition}

\begin{proof}
Note that
\begin{align*}
\bar{S}(\mu_x z_j^n,\mu_x z_j^n,\mu_x z_j^{n+1}) &= S(\mu_x z_j^n) + \frac{1}{3}\nabla S(\mu_x z_j^n)^T (\mu_x z_j^{n+1}-\mu_x z_j^{n})\\
&= S(\mu_x z_j^n) + \frac{1}{3}\nabla_{z_j^n} (S(\mu_x z_j^n))^T (z_j^{n+1}-z_j^{n}),
\end{align*}
and
\begin{equation*}
\begin{split}
&\frac{1}{3}(\delta_{x} z_j^n)^T L_+\mu_x z_j^n +\frac{1}{3}(\delta_{x} z_j^n)^TL_+\mu_x z_j^{n+1} +\frac{1}{3}(\delta_{x}  z_j^{n+1})^TL_+\mu_x z_j^{n}\\
&=(\delta_{x} z_j^{n})^TL_+\mu_x z_j^{n}+\frac{1}{3}\big((\delta_{x} z_j^{n})^TL_+(\mu_x z_j^{n+1}-\mu_x z_j^{n})+(\delta_{x} z_j^{n+1}-\delta_{x} z_j^{n})^TL_+\mu_x z_j^{n}\big)\\
&=(\delta_{x} z_j^{n})^TL_+\mu_x z_j^{n} + \frac{1}{3} \big( (\mu_x z_j^n)^T L_x^T \delta_x + (\delta_x z_j^n)^T L_x \mu_x \big) (z_j^{n+1}-z_j^{n})\\
&=(\delta_{x} z_j^{n})^TL_+\mu_x z_j^{n} + \frac{1}{3} \Big(\nabla_{z_j^n} \big((\delta_{x}z_j^{n})^TL_+\mu_x z_j^{n}\big)\Big)^T (z_j^{n+1}-z_j^{n}).
\end{split}
\end{equation*}
Inserting this in \eqref{eq:localenergy}, we get \eqref{eq:modeneergy} from \eqref{discrete_ECL}. Furthermore, observing that
\begin{equation*}
3\frac {\partial \bar{S}}{\partial x}\bigg\rvert_{(\mu_x z_j^n,\mu_x z_j^{n+1})} = \nabla S(\mu_x z_j^n) + \frac{1}{2} \nabla^2S(\mu_x z_j^n)(\mu_xz_j^{n+1}-\mu_xz_j^n),
\end{equation*}
we may rewrite \eqref{eq:LILEP} as
\begin{equation*}
\Big(K\mu_x+\frac{\Delta t}{2}L\delta_x-\frac{\Delta t}{2}\nabla^2S(\mu_xz_j^n)\mu_x\Big) ( z_j^{n+1}-z_j^{n}) =\Delta t \big(\nabla S(\mu_xz^n_j) - L \delta_x z_j^n\big),
\end{equation*}
which is \eqref{eq:newfrm}.
\end{proof}

Note that \eqref{eq:fullenergy} is the discrete energy preserved by the fully implicit local energy-preserving method of \cite{gong2014some}. Also, for methods based on the multi-symplectic structure, instead of solving for $z$ directly, the normal procedure is to eliminate the auxiliary variables from the scheme and get an equation for one variable $u$. Therefore we do not give an explicit expression for the modified energy in $z^n$. However, in Section \ref{Numerical example}, we present an explicit expression for the modified energy in $u^n$ when our scheme is applied to the KdV equation.
  
The results about the energy conservation for the LILEP method applied to one-dimensional multi-symplectic PDEs can be generalised to problems in spatial dimensions of any finite degree. Consider for example a $2$-dimensional multi-symplectic PDE 
\begin{equation}\label{Multisym_PDEs_2d}
K z_t+L^1z_x+L^2z_y=\nabla S(z), \quad z \in \mathbb{R}^l,\quad (x,y,t)\in \mathbb{R}^3,
      \end{equation}
for which we have the following corollary. This is presented without its proof, which is rather technical but similar to the proof of Theorem \ref{proof of ECL for LECL scheme}.

   \begin{corollary}\label{LEP higher dimension for LILECL scheme}
 The scheme obtained by applying the midpoint rule in space and Kahan's method in time to equation \eqref{Multisym_PDEs_2d},
   \begin{equation}\label{LILECL_Scheme_2d}
K\delta_{t}\mu_{x}\mu_{y}z_{j,k}^n+L^1\delta_x\mu_{t}\mu_yz_{j,k}^n+L^2\delta_y\mu_{t}\mu_xz_{j,k}^n=3\frac {\partial \bar{S}}{\partial x}\bigg\rvert_{(\mu_x\mu_yz_{j,k}^n,\mu_x\mu_yz_{j,k}^{n+1})},
      \end{equation}
 where $j=0,\ldots,M_x-1$ and $k=0,\ldots,M_y-1$, satisfies the discrete local energy conservation law
  \begin{equation*}
\delta_{t} (\bar{E}_L)_{j,k}^n +\delta_{x}{(\bar{F}_L^{1})}_{j,k}^{n}+\delta_{y}{(\bar{F}_L^{2})}_{j,k}^{n}=0,
      \end{equation*}   
      where
  \begin{align*}
  (\bar{E}_L)_{j,k}^n =& \bar{S}(\mu_x\mu_yz_{j,k}^n, \mu_x\mu_yz_{j,k}^n, \mu_x\mu_yz_{j,k}^{n+1})\\
 & +\frac{1}{3}(\delta_{x}\mu_yz_{j,k}^n)^TL^1_+ \mu_x\mu_yz_{j,k}^n +\frac{1}{3}(\delta_{x}\mu_yz_{j,k}^n)^TL^1_+ \mu_x\mu_yz_{j,k}^{n+1}\\
 & + \frac{1}{3}(\delta_{x}\mu_yz_{j,k}^{n+1})^TL^1_+ \mu_x\mu_y z_{j,k}^n + \frac{1}{3}(\delta_{y}\mu_xz_{j,k}^n)^T L^2_+ \mu_x\mu_yz_{j,k}^n \\
 & +\frac{1}{3}(\delta_{y}\mu_xz_{j,k}^n)^T L^2_+ \mu_x\mu_yz_{j,k}^{n+1} + \frac{1}{3}(\delta_{y}\mu_xz_{j,k}^{n+1})^T L^2_+ \mu_x\mu_y z_{j,k}^n,\\
 {(\bar{F}_L^{1})}_{j,k}^{n} =& -\frac{1}{3}(\delta_{t}\mu_yz_{j,k}^n)^TL^1_+ \mu_t \mu_{y} z_{j,k}^n -\frac{1}{3}(\delta_{t}\mu_yz_{j,k}^n)^T L^1_+ \mu_t \mu_{y} z_{j,k}^{n+1} \\
 & -\frac{1}{3}(\delta_{t}\mu_yz_{j,k}^{n+1})^T L^1_+ \mu_{t} \mu_yz_{j,k}^n,\\
 {(\bar{F}_L^{2})}_{j,k}^{n} =& -\frac{1}{3}(\delta_{t}\mu_xz_{j,k}^n)^T L^2_+ \mu_t \mu_{x} z_{j,k}^n-\frac{1}{3}(\delta_{t}\mu_xz_{j,k}^n)^T L^2_+ \mu_t \mu_{x} z_{j,k}^{n+1}\\
 & - \frac{1}{3}(\delta_{t}\mu_xz_{j,k}^{n+1})^T L^2_+ \mu_{t} \mu_xz_{j,k}^n.
      \end{align*}         
  \end{corollary}

\subsection {Global energy-preserving methods for multi-symplectic PDEs}
As shown in Section \ref{multi-symplectic PDES properties}, Hamiltonian PDEs of the form \eqref{Multisym_PDEs} with periodic boundary conditions have global energy conservation which can be deduced from the local conservation law. On the other hand, the local conservation law is not inherent in the global conservation law. In this section, we will focus on giving a systematic method that preserves the global energy conservation law directly. We discretize $\partial_x$ with an antisymmetric differential matrix $D$ and get the semi-discretized variant of \eqref{Multisym_PDEs},
\begin{equation}\label{Multisym_PDEs_semidis_any}
K\partial_tz_j+L(D z)_j =\nabla S(z_j), \quad j=0,1,\ldots,M-1,
      \end{equation}
      where $z:=(z_0,z_1,\ldots,z_{M-1})^T \in \mathbb{R}^{M\times l}$ and $(Dz)_j = \sum_{k=0}^{M-1} D_{j,k} z_k$. We then apply Kahan's method to \eqref{Multisym_PDEs_semidis_any} and obtain the linearly implicit global energy-preserving (LIGEP) scheme
 \begin{equation}\label{eq:LIGEP}
 K\delta_t z_j^n+L(D\mu_tz^{n})_j =3\frac {\partial \bar{S}}{\partial x}\bigg\rvert_{(z_j^n,z_j^{n+1})}.
 \end{equation}  
 Define the polarised energy density by
  \begin{equation}
 \bar{E}_j^n = \bar{S}(z_j^n,z_j^{n},z_j^{n+1})+\frac{1}{3}(Dz^n)_j^TL_+ z_j^{n}+\frac{1}{3}(Dz^n)_j^TL_+ z_j^{n+1}+\frac{1}{3}(Dz^{n+1})_j^TL_+z_j^n,
 \label{eq:globalenergy}
 \end{equation} 
 and we get the following result.
\begin{theorem}\label{th:globalenergylaw}
 For periodic boundary conditions $z(x+P,t)=z(x,t)$, the scheme \eqref{eq:LIGEP} satisfies the discrete global energy conservation law
  \begin{equation}
 \bar{\mathcal{E}}^{n+1}= \bar{\mathcal{E}}^{n}, \quad \bar{\mathcal{E}}^{n}:=\Delta x\sum_{j=0}^{M-1} \bar{E}_{j}^n, \quad \Delta x = P/M.
 \label{eq:polglobenergy}
  \end{equation}
  \end{theorem}   
 \begin{proof}
Taking the inner product with $\frac{1}{3}\delta_tz_j^n$ on both sides of equation \eqref{eq:LIGEP} and using the skew-symmetry of the matrix $K$, we get
    \begin{align}\label{inner product1}
\frac{1}{3} (\delta_tz_j^n)^TL(D\mu_tz^{n})_j= (\delta_tz_j^n)^T\frac {\partial \bar{S}}{\partial x}\bigg\rvert_{(z_j^n,z_j^{n+1})}.
      \end{align}
      Taking the inner product with $\frac{1}{3}\delta_tz_j^{n+1}$ on both sides of \eqref{eq:LIGEP}, we get
    \begin{align}\label{inner product2}
\frac{1}{3} (\delta_tz_j^{n+1})^T K \delta_tz_j^{n} + \frac{1}{3} (\delta_tz_j^{n+1})^TL(D\mu_tz^{n})_j= (\delta_tz_j^{n+1})^T\frac {\partial \bar{S}}{\partial x}\bigg\rvert_{(z_j^n,z_j^{n+1})}.
      \end{align} 
    Furthermore, taking the inner product with $\frac{1}{3}\delta_tz_j^n$ on both sides of \eqref{eq:LIGEP} for the next time step, we have
     \begin{align}\label{inner product3}
\frac{1}{3} (\delta_tz_j^{n})^T K \delta_tz_j^{n+1} + \frac{1}{3} (\delta_tz_j^n)^T L (D \mu_tz^{n+1})_j= (\delta_tz_j^n)^T\frac {\partial \bar{S}}{\partial x}\bigg\rvert_{(z_j^{n+1},z_j^{n+2})}.
      \end{align}
Adding equations \eqref{inner product1}, \eqref{inner product2} and \eqref{inner product3}, we get 
\begin{equation}\label{inner product}
\begin{split}
\frac{1}{3} \big((\delta_tz_j^n)^T L (D\mu_tz^{n})_j+(\delta_tz_j^{n})^T L (D\mu_tz^{n+1})_j&\\
+ (\delta_tz_j^{n+1})^T L (D\mu_tz^{n})_j\big)& =\delta_t \bar{S}(z_j^n,z_j^n,z_j^{n+1}).
\end{split}
\end{equation}
  By using the commutative laws and discrete Leibniz rules,
   \begin{equation}\label{using discrete Leibnize rule}
   		\begin{split}
   		\delta_t((Dz^{n})_j^TL_+z_j^n)&=(D\delta_tz^{n})_j^T  L_+\mu_tz_j^n+(D\mu_tz^{n})_jL_+ \delta_t z_j^n,\\
   		\delta_t((Dz^{n})_j^TL_+z_j^{n+1})&=(D\delta_tz^{n})_j^T  L_+\mu_tz_j^{n+1}+(D\mu_tz^{n})_jL_+ \delta_t z_j^{n+1},\\
   		\delta_t((Dz^{n+1})_j^TL_+z_j^n)&=(D\delta_tz^{n+1})_j^T  L_+\mu_tz_j^n+(D\mu_tz^{n+1})_jL_+ \delta_t z_j^n.
        \end{split} 
   \end{equation} 
  Based on the above equations \eqref{inner product} and \eqref{using discrete Leibnize rule}, we obtain
\begin{equation*}
     \begin{split}
   \delta_t E_j^n =& \,  \delta_t \bar{S}(z_j^n,z_j^n,z_j^{n+1})+ \frac{1}{3}\big(\delta_t((Dz^{n})_j^TL_+z_j^n)+(Dz^{n})_j^TL_+z_j^{n+1}+(Dz^{n+1})_j^TL_+z_j^n\big)\\
   = & \, \frac{1}{3}\big((\delta_tz_j^n)^TL_+(D\mu_tz^{n})_j+(D \delta_tz^{n})_j^TL_+\mu_tz_j^n\big)\\
   & \, +\frac{1}{3}\big((\delta_tz_j^{n+1})^TL_+(D\mu_tz^{n})_j+(D \delta_tz^{n+1})_j^TL_+\mu_tz_j^n\big)\\
   & \, +\frac{1}{3}\big((\delta_tz_j^{n})^TL_+(D\mu_tz^{n+1})_j+(D \delta_tz^{n})_j^TL_+\mu_tz_j^{n+1}\big)\\
   = & \,\sum_{k=0}^{N-1}(D)_{j,k}G_{j,k},
    \end{split}
\end{equation*}
where 
 \begin{align*}
G_{j,k}:=&\frac{1}{3}\big((\delta_tz^n)_j^TL_+\mu_tz_L^n+(\delta_tz^n)_L^TL_+\mu_tz_j^n\big)\\
&+\frac{1}{3}\big((\delta_tz^{n+1})_j^TL_+\mu_tz_L^n+(\delta_tz^{n+1})_L^TL_+\mu_tz_j^n\big)\\
&+\frac{1}{3}\big((\delta_tz^n)_j^TL_+\mu_tz_L^{n+1}+(\delta_tz^n)_L^TL_+\mu_tz_j^{n+1}\big).
\end{align*}
Since $D$ is skew-symmetric and $G_{j,k}=G_{k,j}$, we get 
$$\sum_{j=0}^{M-1}\delta_t \bar{E}_j^n=0,$$
which implies that the discrete global energy conservation law
 $\bar{\mathcal{E}}^{n+1}= \bar{\mathcal{E}}^{n}$ is satisfied.
  \end{proof} 
The polarised energy $\bar{\mathcal{E}}$ preserved by \eqref{eq:LIGEP} may also be expressed as a modification of the discrete energy 
\begin{align}
\mathcal{E}^{n} := \Delta x\sum_{j=0}^{M-1} E(z_{j}^n), \quad E(z_{j}^n) = S(z_j^n) + (Dz^n)_j^TL_+ z_j^{n},
\label{fullglobalenergy}
\end{align}
which is preserved by the fully implicit global energy-preserving scheme of \cite{gong2014some}. The proof of the following proposition is similar to the proof of Proposition \ref{prop:modE}, and hence omitted.

\begin{proposition}\label{prop:modEfull}
If the solution $z^{n+1}$ is found from $z^n$ by \eqref{eq:LIGEP}, 
the discrete global energy $\bar{\mathcal{E}}^{n}$ of \eqref{eq:polglobenergy}  satisfies
\begin{equation*}
\bar{\mathcal{E}}^{n} = \mathcal{E}^{n} + \Delta x \sum_{j=0}^{M-1} \frac{1}{3}(\nabla E(z_j^n))^T (z_j^{n+1}-z_j^n),
\end{equation*}
and $z_j^{n+1}-z_j^n$ satisfies
\begin{align*}\label{eq:newfrmglobal}
R(z_j^n)(z_j^{n+1}-z_j^n)=\Delta t g(z_j^n),
\end{align*}
where $g(z_j^n)=\nabla S(z_j^n)-L(Dz)_j^n$ and $R(z_j^n)=K+\frac{\Delta t}{2}\nabla g(z_j^n)$.
\end{proposition}

The above global conservation results can be generalised to multi-symplectic formulations in higher spatial dimensions, as demonstrated for the two-dimensional case by the following corollary, whose omitted proof is in the same vein as the proof of Theorem \ref{th:globalenergylaw}.
 
\begin{corollary}\label{GEP higher dimension for LIGECL scheme}
 Discretizing $\partial_x$ and $\partial_y$ by skew-symmetric differential matrices $D_x$ and $D_y$ in equation \eqref{Multisym_PDEs_2d} and then applying Kahan's method to the semi-discrete system, one obtains the linearly implicit global energy-preserving (LIGEP) scheme
   \begin{equation}\label{LIGECL_Scheme_2d}
K\delta_{t}z_{j,k}^n+L^1\mu_{t}(D_xz^{n})_{j,k}+L^2\mu_{t}(D_yz^{n})_{j,k}=3\frac {\partial \bar{S}}{\partial x}\bigg\rvert_{(z_{j,k}^n,z_{j,k}^{n+1})},
      \end{equation}
where $j=0,\ldots,M_x-1$ and $k=0,\ldots,M_y-1$.
For periodic boundary conditions $z(x+P_x,y,t)=z(x,y,t)$, $z(x,y+P_y,t)=z(x,y,t)$, the scheme \eqref{LIGECL_Scheme_2d} satisfies the discrete global energy conservation law
  \begin{equation*}
 \bar{\mathcal{E}}^{n+1}= \bar{\mathcal{E}}^{n},
      \end{equation*}     
      where
  \begin{align*}
  \bar{\mathcal{E}}^{n}:= & \,\Delta x \, \Delta y \sum_{j=0}^{M_x-1} \sum_{k=0}^{M_y-1} \bar{E}_{j,k}^n, \quad \Delta x = P_x/M_x, \quad \Delta y = P_y/M_y, \\
 \bar{E}_{j,k}^n = & \, \bar{S}(z_{j,k}^n,z_{j,k}^{n},z_{j,k}^{n+1})\\
 & \, +\frac{1}{3}(D_xz^{n})_{j,k}^TL^1_+ z_{j,k}^{n}+\frac{1}{3}(D_xz^{n})_{j,k}^TL^1_+ z_{j,k}^{n+1}+\frac{1}{3}(D_xz^{n+1})_{j,k}^TL^1_+z_{j,k}^n,\\
& \, +\frac{1}{3}(D_yz^{n})_{j,k}^TL^2_+ z_{j,k}^{n}+\frac{1}{3}(D_yz^{n})_{j,k}^TL^2_+ z_{j,k}^{n+1}+\frac{1}{3}(D_yz^{n+1})_{j,k}^TL^2_+z_{j,k}^n.
 \end{align*} 
\end{corollary}

\section{Numerical examples}\label{Numerical example}
In this section, we apply our proposed new linearly implicit energy-preserving schemes to the KdV equation and Zakharov--Kuznetsov equation, and compare them with fully implicit schemes. Among our reference methods are the methods introduced in  \cite{gong2014some}, for which the local energy-preserving method is denoted by LEP, and the global energy-preserving method by GEP. These schemes are discretized in space the same way as our LILEP and LIGEP schemes, but the fully implicit AVF method is used for the time-stepping. For the GEP and LIGEP schemes, two different choices are considered for approximating the spatial derivative: the central difference operator $\delta^c_x$ defined by $\delta^c_x v_j^n := \frac{1}{2}(\delta_x v_{j-1}^n+\delta_x v_j^n)$ and the first order Fourier pseudospectral operator \cite{bridges2001multi}. The latter results in the $M \times M$ matrix $D$, given explicitly by its elements
\begin{equation*}
D_{i,j} = 
\begin{cases}
    \frac{\pi}{P} (-1)^{i+j} \cot{\left(\pi (i-j)/M\right)},     & \quad \text{if} \quad i \neq j,\\
    0,     & \quad \text{if} \quad i = j,
\end{cases}
\end{equation*}
evaluated on the domain $\left[0,P\right]$, where we assume $M$ even and periodic boundary conditions \cite{chen2011multi}. If $M$ is odd, we have instead
\begin{equation*}
D_{i,j} = 
\begin{cases}
    \frac{\pi}{P} (-1)^{i+j} \cot{\left(\pi (i-j)/M\right)},     & \quad \text{if} \quad \lvert i-j\rvert < M/2,\\
    \frac{\pi}{P} (-1)^{i+j} \cot{\left(\pi (j-i)/M\right)},     & \quad \text{if} \quad \lvert i-j\rvert > M/2,\\
    0,     & \quad \text{if} \quad i = j.
\end{cases}
\end{equation*}

The numerical results presented in this section are obtained from schemes implemented in MATLAB (2018b release), running on an early 2015 MacBook Pro with a dual-core 3.1 GHz Intel Core i7 processor and 16 GB of 1867 MHz DDR3 RAM. All fully implicit schemes are solved at each step by Newton's method until $\lVert F(u^{n}) \rVert_2 < 10^{-10}$. Linear systems are solved using the backslash operator of MATLAB. MATLAB and Python codes for the experiments are available at \url{https://doi.org/10.5281/zenodo.3709463}.

\subsection{Korteweg--de Vries equation}\label{Numerical example_KdV}
Consider the multi-symplectic structure of the KdV equation as presented in Example \ref{ex:KdV}.
Applying the LILEP scheme \eqref{eq:LILEP} to \eqref{Kdv_multisym_element}, we obtain
\begin{equation*}\label{Kdv_multisym_element_LILEP}
\begin{split}
\frac{1}{2}\delta_t\mu_xu_j^n+\delta_x\mu_tw_j^n &=0,\\
-\frac{1}{2}\delta_t\mu_x\phi_j^n-\gamma \delta_x\mu_tv_j^n &=-\mu_t\mu_xw_j^n+\frac{\eta}{2}\mu_xu_j^n\mu_xu_j^{n+1},\\
\gamma \delta_x\mu_tu_j^n &=\mu_t\mu_xv_j^n,\\
\delta_x\mu_t\phi_j^n &=\mu_t\mu_xu_j^n.
\end{split}
\end{equation*}
By eliminating the auxiliary varibles $\phi, v$ and $w$, we see that this is equivalent to
\begin{equation*}
\delta_t\mu_t\mu_x^3u_j^n+\frac{\eta}{2}\delta_x\mu_t\mu_x(\mu_xu_j^n\mu_xu_j^{n+1})+\gamma^2\delta_x^3\mu_t^2u_j^n=0.
\end{equation*}
Omitting the average operator $\mu_t$ gives us
\begin{equation}\label{eq_KdV_pointwise_IMLECscheme_2 time level}
\delta_t\mu_x^3u_j^n+\frac{\eta}{2}\delta_x\mu_x(\mu_xu_j^n\mu_xu_j^{n+1})+\gamma^2\delta_x^3\mu_tu_j^n=0.
\end{equation}
The polarised discrete energy preserved by this scheme is
\begin{equation} \label{discrete global energy from local energy kdv}
\bar{\mathcal{E}}_L^n=\Delta x\sum_{j=0}^{M-1}\Big(-\frac {1}{6}\gamma^2\big((\delta_x u_j^n)^2+2\delta_x u_j^n\delta_x u_j^{n+1}\big)+\frac {1}{6}\eta\big(\mu_x u_j^n)^2\mu_x u_j^{n+1} \Big).
\end{equation}
On the other hand, the discrete energy preserved by the LEP method of \cite{gong2014some} is
\begin{equation} \label{discrete global energy from local energy kdv_LECpaper}
\mathcal{E}_L^n=\Delta x\sum_{j=0}^{M-1}\Big(-\frac {1}{2}\gamma^2(\delta_x u_j^n)^2+\frac {1}{6}\eta(\mu_x u_j^n)^3 \Big).
\end{equation}
By Proposition \ref{prop:modE} and elimination of the variables $\phi, v$ and $w$, \eqref{discrete global energy from local energy kdv} can be expressed as a modification of \eqref{discrete global energy from local energy kdv_LECpaper}:
we may rewrite \eqref{eq_KdV_pointwise_IMLECscheme_2 time level} as
\begin{equation*}
u_j^{n+1} - u_j^n = - \Delta t \, \big(\mu_x^3 + \frac{\Delta t}{2}\gamma^2\delta_x^3 + \frac{\Delta t}{2}\eta \delta_x \mu_x \text{diag}{(\mu_x u_n)}\mu_x\big)^{-1}\big(\gamma^2 \delta_x^3 u^n + \frac{\eta}{2} \delta_x \mu_x (\mu_x u^n)^2\big),
\end{equation*}
where $(\mu_x u^n)^2$ denotes the element-wise square of $\mu_x u^n$.
Inserting this in \eqref{discrete global energy from local energy kdv}, we get 
\begin{align*}
\bar{\mathcal{E}}_L^n = & \, \mathcal{E}_L^n - \frac{\Delta t \, \Delta x}{3} \big(-\gamma^2 \delta_x^T \delta_x u^n + \frac{\eta}{2} \mu_x^T (\mu_x u^n)^2\big)^T \\
& \, \big( \mu_x^3 + \frac{\Delta t}{2} \gamma^2 \delta_x^3 + \frac{\Delta t}{2} \eta \delta_x \mu_x \text{diag}(\mu_x{u^n}) \mu_x\big)^{-1} \big(\gamma^2 \delta_x^3 u^n + \frac{\eta}{2}\delta_x \mu_x (\mu_x u^n)^2\big)\\
= & \, \mathcal{E}_L^n + \frac{\Delta t}{3} (\nabla \mathcal{E}_L^n)^T \big( \mu_x^3 - \frac{\Delta t}{2} \zeta_L'(u^n)\big)^{-1} \zeta_L(u^n),
\end{align*}
with
\begin{equation*}
\zeta_L(u^n) =  - \gamma^2 \delta_x^3 u^n -\frac{\eta}{2} \delta_x \mu_x (\mu_x u^n)^2,
\end{equation*}
where $\nabla \mathcal{E}_L^n$ means the gradient of $\mathcal{E}_L^n$ with respect to $u^n$, and $\zeta_L'(u^n)$ denotes the Jacobian matrix of $\zeta_L (u^n)$.

Similarly for the LIGEP method \eqref{eq:LIGEP}; applying it to the the multi-symplectic KdV equations \eqref{Kdv_multisym_element} and eliminating the auxiliary varibles $\phi, v$ and $w$, we obtain 
\begin{equation*}
\delta_t \mu_t u_j^n+\frac{\eta}{2}\mu_t(D(u^nu^{n+1}))_j+\gamma^2\mu_t^2(D^3u^n)_j=0,
\end{equation*}
where $u^nu^{n+1}$ denotes element-wise multiplication of the vectors.
Omitting the average operator $\mu_t$, we get
\begin{equation}\label{eq_KdV_pointwise_IMGECscheme_2 time level}
\delta_tu_j^n+\frac{\eta}{2}(D(u^nu^{n+1}))_j+\gamma^2\mu_t(D^3u^n)_j=0.
\end{equation}
The discrete global energy preserved by the GEP method is
\begin{equation} \label{discrete global energy from global energy kdv_LECpaper}
\mathcal{E}^n=\Delta x\sum_{j=0}^{M-1}\Big(-\frac {1}{2}\gamma^2(Du^n)_j^2+\frac {1}{6}\eta(u_j^n)^3 \Big),
\end{equation}
while the polarised discrete energy preserved by \eqref{eq_KdV_pointwise_IMGECscheme_2 time level} is
\begin{equation}\label{discrete global energy from global energy kdv}
\begin{split}
\bar{\mathcal{E}}^n = & \, \Delta x\sum_{j=0}^{M-1}\Big(-\frac {1}{6}\gamma^2\big((Du^n)_j^2+2(Du^n)_j(D u^{n+1})_j\big)+\frac {1}{6}\eta\big(u_j^n)^2 u_j^{n+1} \Big) \\
= & \, \mathcal{E}^n - \frac{\Delta t \, \Delta x}{3} \big(-\gamma^2 D^T D u^n + \frac{\eta}{2} (u^n)^2\big)^T \\
 & \, \big( I + \frac{\Delta t}{2} \gamma^2 D^3 + \frac{\Delta t}{2} \eta D \, \text{diag}(u^n)\big)^{-1} \big(\gamma^2 D^3 u^n + \frac{\eta}{2}D (u^n)^2\big)\\
=  & \, \mathcal{E}^n + \frac{\Delta t}{3} (\nabla\mathcal{E}^n)^T \big( I - \frac{\Delta t}{2} \zeta'(u^n)\big)^{-1} \zeta(u^n),
\end{split}
\end{equation}
where $\zeta(u^n) = -\gamma^2 D^3 u^n - \frac{\eta}{2}D (u^n)^2$.

\subsubsection*{Test problem 1}
In the first numerical experiment, we consider the problem introduced in \cite{zabusky1965interaction} and then used by Zhao and Qin \cite{zhao2000multisymplectic} and Ascher and McLachlan \cite{ascher2005symplectic} to test various symplectic and multi-symplectic schemes: the KdV equation with $\gamma=0.022$, $\eta=1$, and initial value
$$u_0(x)=\cos(\pi x),$$
with $x\in[0,P]$, $P=2$. 
This problem is also considered in Example 3 of \cite{gong2014some}, where it is solved by implicit schemes that preserve local and/or global energy.
As observed by Gong et al.\, the global energy-preserving scheme (GEP) with the central difference operator used to approximate $\partial_x$ gives unsatisfactory results for this problem; we observed that the same is true for the LIGEP scheme. Therefore, the Fourier pseudospectral operator is used to approximate the spatial derivatives in the GEP and LIGEP schemes. This seems to result in more accurate solutions than the LEP and LILEP schemes for the same number of discretization points, but at a considerably higher computational cost, as seen from Table \ref{tab:timecost_cospix}. 
Also from Table \ref{tab:timecost_cospix}, we see an example of the advantage that can be gained by having a linearly implicit scheme instead of a fully implicit scheme. The different running times give an indication of the number of iterations necessary in Newton's method to solve the fully implicit schemes for the particular cases.
From Figure \ref{KdV_cospix}, we can conclude that our linearly implicit schemes give results close to their fully implicit counterparts introduced in \cite{gong2014some}, and that the different schemes converge to the same solution.

\begin{table}[ht]
{
\caption{Computational time, in seconds, for finding the solution of the first test problem at time $t=5$ by a temporal step size $\Delta t=0.005$ and various number of discretization points in space, $M$.}
\label{tab:timecost_cospix}
\begin{center}
\begin{tabular}{lccccccc}
\hline
\multicolumn{1}{c}{$M$} & $200$ & $400$ & $600$ & $800$ & $1000$ & $1500$ & $2000$\\
\hline
LEP      	& 1.87 & 3.16 & 4.43 				& 11.18 & 13.81 & 21.53 & 28.54     \\
LILEP    	& 4.24e-1 & 7.40e-1 & 1.07    	& 1.39  & 1.73 & 2.67 & 3.58  \\
GEP 			& 12.29 & 78.11 & 242.48   		& 1016.57 & 1888.69 & 5793.18 & 13154.20     \\
LIGEP 		& 2.16 & 11.15 & 33.50     		& 73.94  & 136.93  & 398.53  & 894.52   \\
\hline
\end{tabular}
\end{center}
}
\end{table}

\begin{figure}[ht]
		\centering
	  \begin{minipage}{.47\textwidth}
		\centering
                \includegraphics[width=0.99\linewidth]{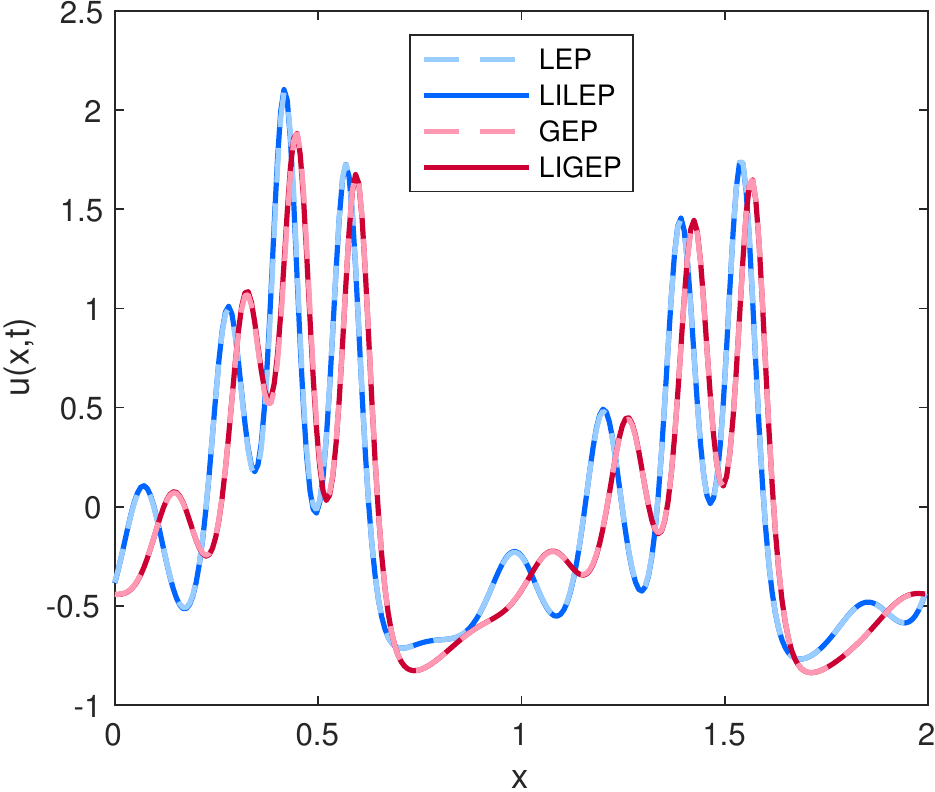}
        \end{minipage}\hspace{14pt}
          \begin{minipage}{.47\textwidth}
        \centering
                \includegraphics[width=0.99\linewidth]{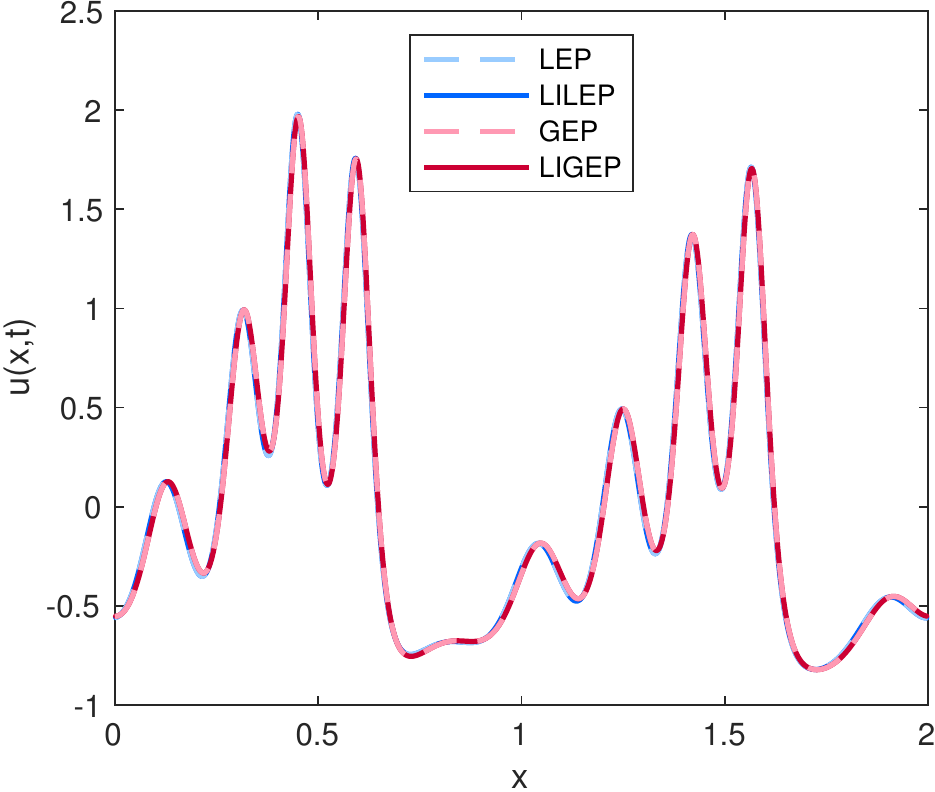}
        \end{minipage}
      \caption{Solution of test problem 1 at time $t=5$ by our schemes and the fully implicit schemes of Gong et al. \textit{Left:} $M=250$, $\Delta t = 0.02$.  \textit{Right:} $M=1000$, $\Delta t = 0.002$.}
      \label{KdV_cospix}
\end{figure}

Compared to the schemes tested in \cite{zhao2000multisymplectic,ascher2005symplectic}, our schemes do also perform well; see Figure \ref{KdV_cospix_AM}, where we have plotted solutions by our schemes for the same discretization parameters used in Example 5.3 of \cite{ascher2005symplectic}. The reference solution is found by the implicit midpoint scheme of \cite{ascher2005symplectic} with $\theta = 1$ and very fine discretization in space and time: $M = 2000$ and $\Delta t = 0.0001$. We observe that the LILEP scheme behaves similarly to the multi-symplectic box scheme of Arscher and McLachlan (see figures 3 and 4 in \cite{ascher2005symplectic}), seemingly with the same superior stability for rough discretization in space and time. The LIGEP scheme, on the other hand, starts to blow up at around $t=1$ when $M=60$, $\Delta t=1/150$, but produces for $M=100$, $\Delta t=0.004$ a solution that is much closer to the correct solution than any of the schemes tested in \cite{ascher2005symplectic} (see Figure 3 in that paper for comparison).

\begin{figure}[ht]
		\centering
	  \begin{minipage}{.47\textwidth}
		\centering
                \includegraphics[width=0.99\linewidth]{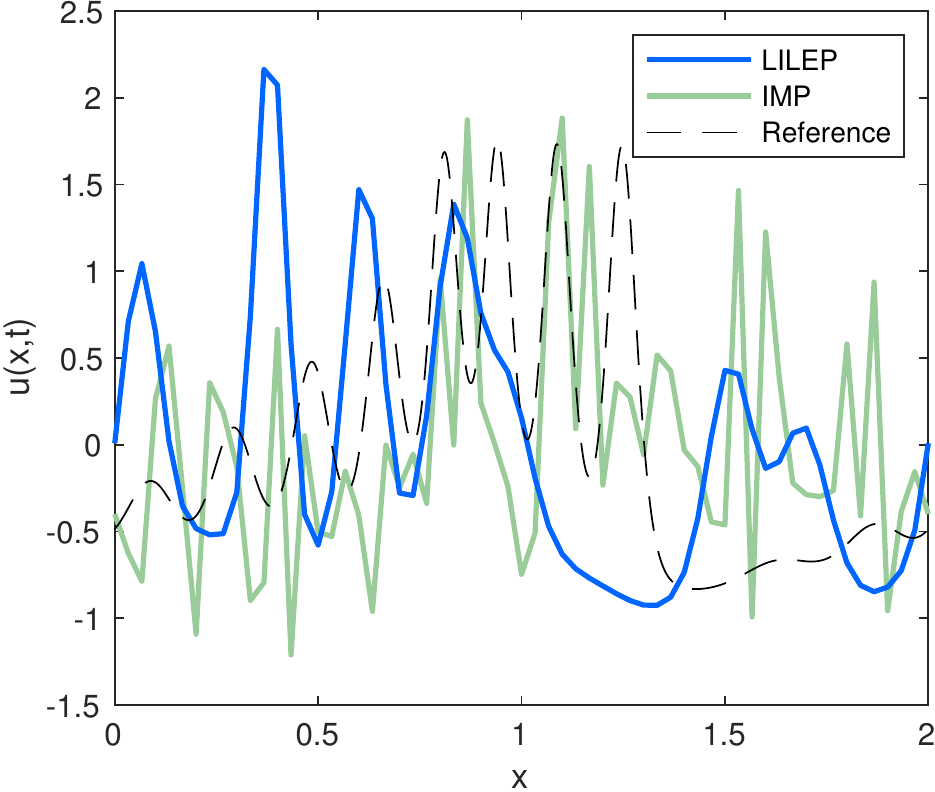}
        \end{minipage}\hspace{14pt}
          \begin{minipage}{.47\textwidth}
        \centering
                \includegraphics[width=0.99\linewidth]{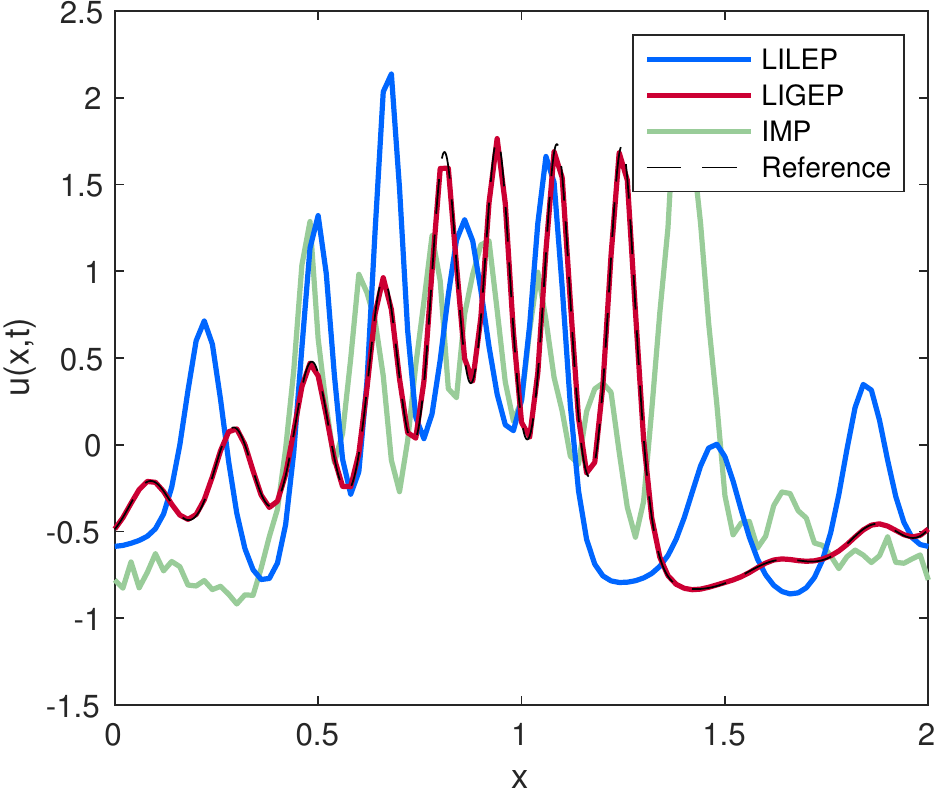}
        \end{minipage}
      \caption{Solutions of test problem 1 at time $t=10$ by our schemes and the implicit midpoint scheme (IMP) as given in \cite{ascher2005symplectic} (with $\theta = 2/3$ in the left figure and $\theta = 1$ in the right figure). \textit{Left:} $M=60$, $\Delta t = 1/150$. \textit{Right:} $M=100$, $\Delta t = 0.004$.}
      \label{KdV_cospix_AM}
\end{figure}

\subsubsection*{Test problem 2}
To get quantitative results on the performance of our methods, we wish to study a problem with a known solution. For the KdV equation with $\gamma=1$, $\eta=6$, initial value $u_0(x)= \frac{1}{2} c \, \mathrm{sech}^2(-x+P/2)$ and periodic boundary conditions $u(x+P,t) = u(x,t)$, the exact solution is a soliton moving with a constant speed $c$ in the positive $x$-direction while keeping its initial shape. That is,
\begin{equation*}
u(x,t)= \frac{1}{2} c \, \mathrm{sech}^2((-x+c t) \text{ }\mathrm{mod} \text{ } P-P/2).
\end{equation*}
In our numerical experiments, $c=4$ and $P=20$. For this problem, we have used the central difference operator to approximate $\partial_x$ in the GEP and LIGEP schemes, since it gives good results and yields considerably shorter computational time than if the pseudospectral operator is used. The proposed methods all show very good stability conditions when applied to this problem, as expected by methods conserving some invariant. The initial shape of the wave is well kept for long integration times, even when quite large step sizes in space and time are used; Figure \ref{KdV_soliton} gives a good illustration of this. As in the previous example, we again observe that little is lost in accuracy by choosing linearly implicit over fully implicit time integration. A close inspection of Figure \ref{KdV_soliton} also indicates that the local energy-preserving schemes preserve the shape of the wave better than the global energy-preserving schemes, while on the other hand, the GEP and LIGEP schemes are better than the LEP and LILEP schemes at preserving the speed of the wave. This is confirmed in Table \ref{tab:error_vs_dx} by measuring the shape error
$$\epsilon_{\text{shape}}:=\underset{\tau}{\text{min}}\parallel U^N-u(\cdot-\tau)\parallel_2^2$$
and phase error
$$\epsilon_{\text{phase}}:= \lvert\underset{\tau}{\text{argmin}}\parallel U^N-u(\cdot-\tau)\parallel_2^2-c t\rvert,$$
where $U^N$ is the numerical solution at end time $t$.

\begin{figure}[ht]
\centering
      \includegraphics[width=0.66\textwidth]{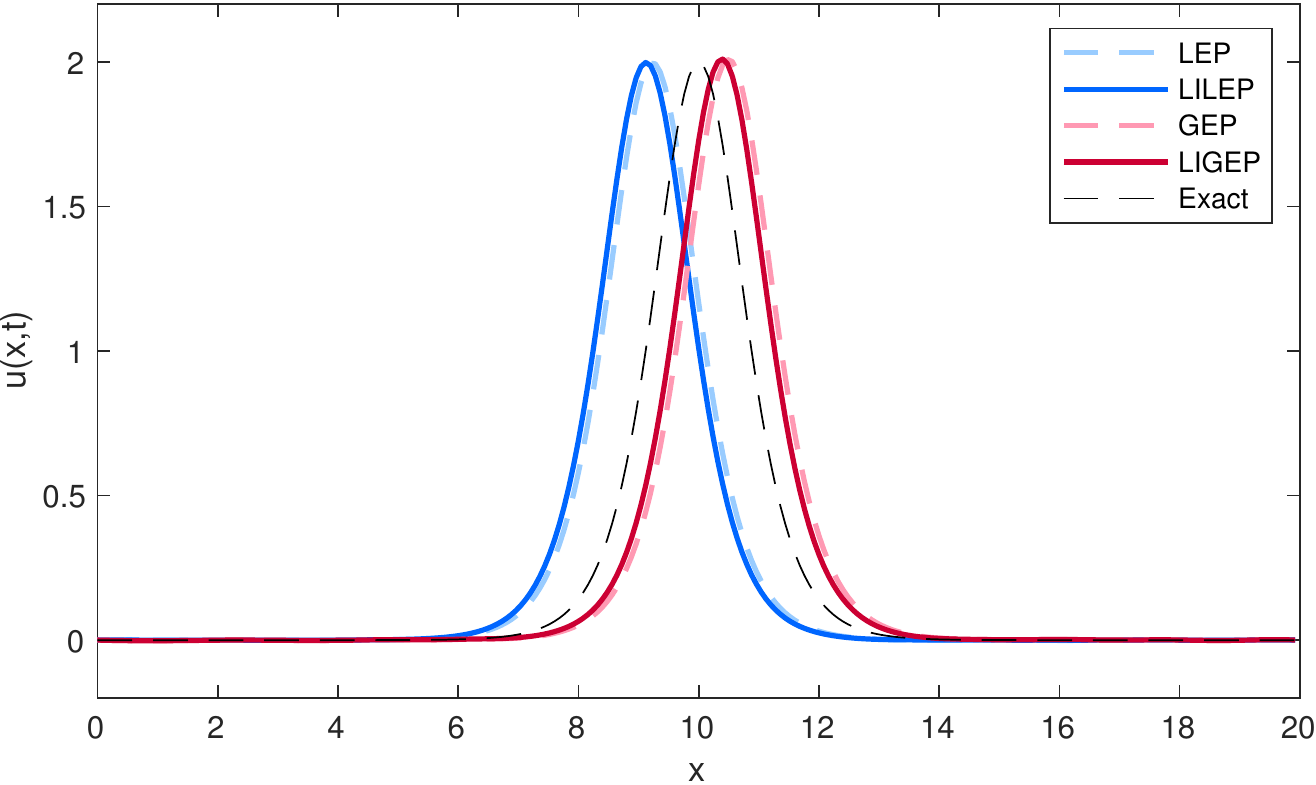}
      \caption{The soliton solution of the KdV equation at time $t=100$, with $M = 250$ discretization points in space and a time step $\Delta t = 0.01$.}\label{KdV_soliton}
\end{figure}

\begin{table}[!ht]
{
\caption{Phase and shape errors and the computational time (CT) for different schemes applied to test problem 2 of the KdV equation, for varying number of discretization points $M$, with time step $\Delta t = 0.01$ and end time $t=100$.}
\label{tab:error_vs_dx}
\begin{center}
\resizebox{\columnwidth}{!}{
\begin{tabular}{lccccccccc}
\multicolumn{1}{c}{$M$} & \multicolumn{3}{c}{200} & \multicolumn{3}{c}{400} & \multicolumn{3}{c}{600} \\
\hline
& $\epsilon_\text{shape}$ & $\epsilon_\text{phase}$ & CT & $\epsilon_\text{shape}$ & $\epsilon_\text{phase}$ & CT & $\epsilon_\text{shape}$ & $\epsilon_\text{phase}$ & CT\\
\hline
LEP      	& 4.67e-3  	& 1.12 & 21.86 & 1.22e-3 & 3.81e-1 & 35.89 & 5.86e-4 & 2.43e-1 & 51.92      \\
LILEP    	& 4.10e-3   	& 1.23 & 5.14 & 5.26e-4 & 4.88e-1 & 8.26 & 1.45e-4 & 3.50e-1 & 10.89  \\
GEP 				& 1.62e-2   	& 8.61e-1  & 19.53 & 3.66e-3 & 1.16e-1 & 34.09 & 1.71e-3 & 2.32e-2 & 49.45    \\
LIGEP 			& 1.71e-2    	& 7.50e-1  & 6.84 & 4.39e-3 & 5.19e-5 & 8.10 & 2.47e-3 & 1.31e-1 & 12.52    \\
\hline
\end{tabular}
}
\end{center}
}
\end{table}

In Figure \ref{KdV_timecost}, we have plotted the computational time required to reach a certain accuracy in the global error for the different methods, both at time $t=0.5$ and at time $t=10$. We compare our methods to the fully implicit LEP and GEP schemes of \cite{gong2014some}, to a scheme based on discretizing the standard form \eqref{eq_KdV} of the KdV equation in space and applying the AVF method in time (AVFM), and also to two of the schemes studied in \cite{ascher2005symplectic}: the multi-symplectic box scheme (MSB) and the implicit midpoint scheme (IMP). Most notably we see from both plots in Figure \ref{KdV_timecost} that the linearly implicit schemes perform better than the fully implicit schemes. Also, we see that at time $t=0.5$ the global error is lowest for the LILEP scheme, while at $t=10$ it is lowest for the LIGEP scheme. This is in accordance with the schemes' phase and shape errors, which can be observed from Figure \ref{KdV_soliton} and Table \ref{tab:error_vs_dx}; with increasing time, the phase error becomes more dominant, and thus the scheme with the smallest phase error becomes increasingly advantageous.

\begin{figure}[ht]
		\centering
	  \begin{minipage}{.47\textwidth}
		\centering
                \includegraphics[width=0.99\linewidth]{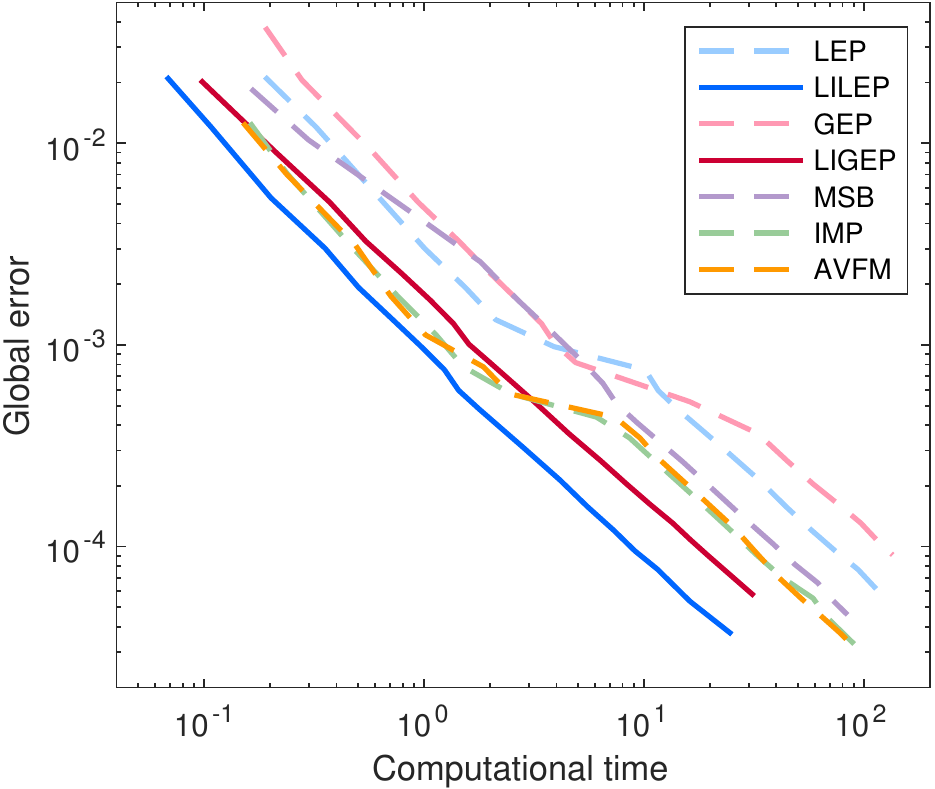}
        \end{minipage}\hspace{14pt}
          \begin{minipage}{.47\textwidth}
        \centering
                \includegraphics[width=0.99\linewidth]{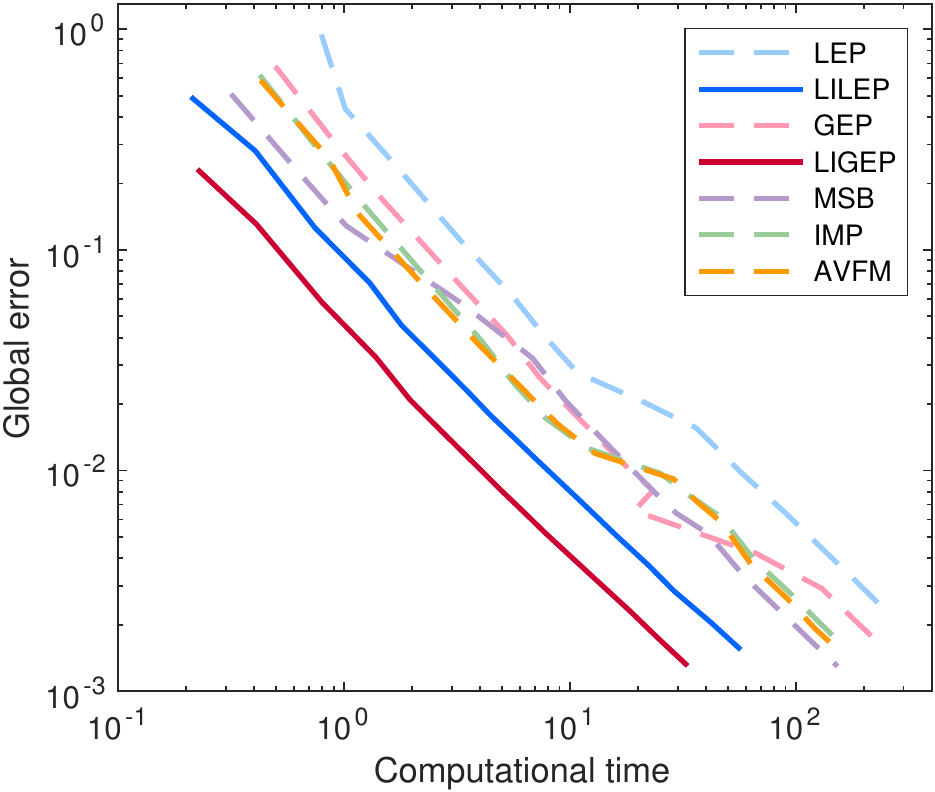}
        \end{minipage}
      \caption{Computational time required to reach a given global error, with $\frac{\Delta x}{\Delta t}$ fixed, for test problem 2 of the KdV equation solved at time $t$. \textit{Left:} $t = 0.5$, $\frac{\Delta x}{\Delta t} = 40$. \textit{Right:} $t = 10$, $\frac{\Delta x}{\Delta t} = 8$.}
      \label{KdV_timecost}
\end{figure}

Figure \ref{KdV_soliton_energies} illustrates how the different schemes preserve a discrete approximation to the energy to machine precision. That is, the linearly implicit schemes LILEP and LIGEP preserve exactly the discrete energies \eqref{discrete global energy from local energy kdv} and \eqref{discrete global energy from global energy kdv}, respectively, while keeping the discrete energies \eqref{discrete global energy from local energy kdv_LECpaper} and \eqref{discrete global energy from global energy kdv_LECpaper}, respectively, within some bound which depends on the discretization parameters. Likewise, the reverse is true for the fully implicit schemes. These observations fit well with our above results about the different discrete approximations to the energy: that for both the local energy preserving and the global energy preserving schemes, either discrete energy given can be seen as a modification of the other approximation. Finally, we have included plots in Figure \ref{KdV_orderplots} which confirm that our schemes are of second order in space and time.

\begin{figure}[ht]
		\centering
	  \begin{minipage}{.47\textwidth}
		\centering
                \includegraphics[width=0.99\linewidth]{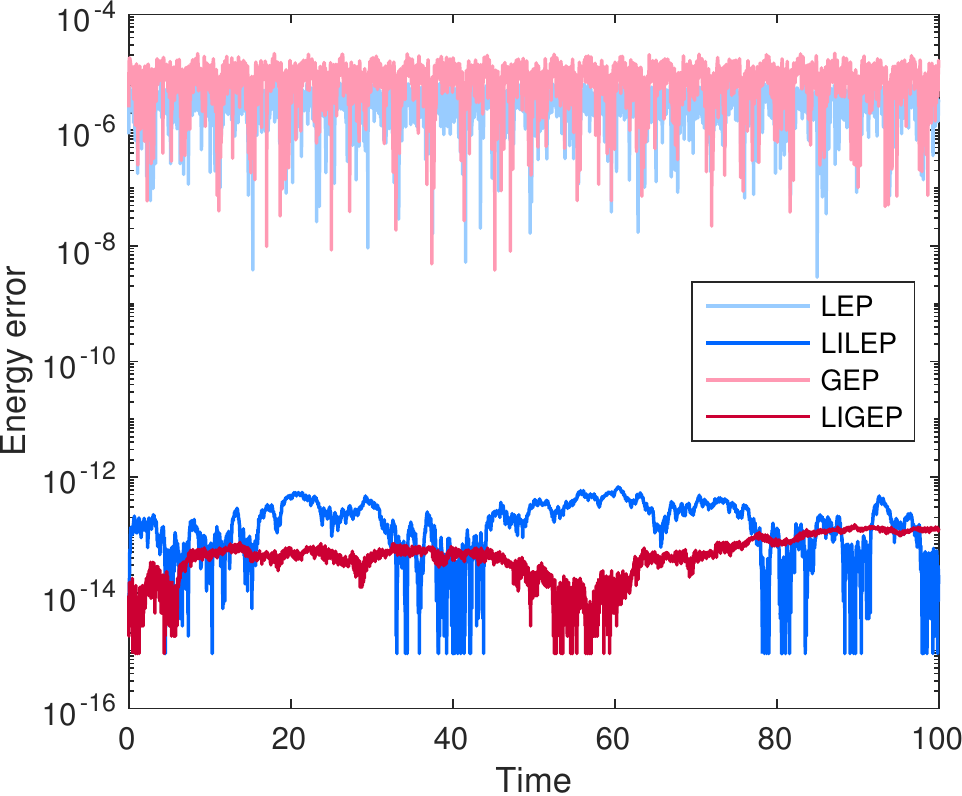}
        \end{minipage}\hspace{14pt}
          \begin{minipage}{.47\textwidth}
        \centering
                \includegraphics[width=0.99\linewidth]{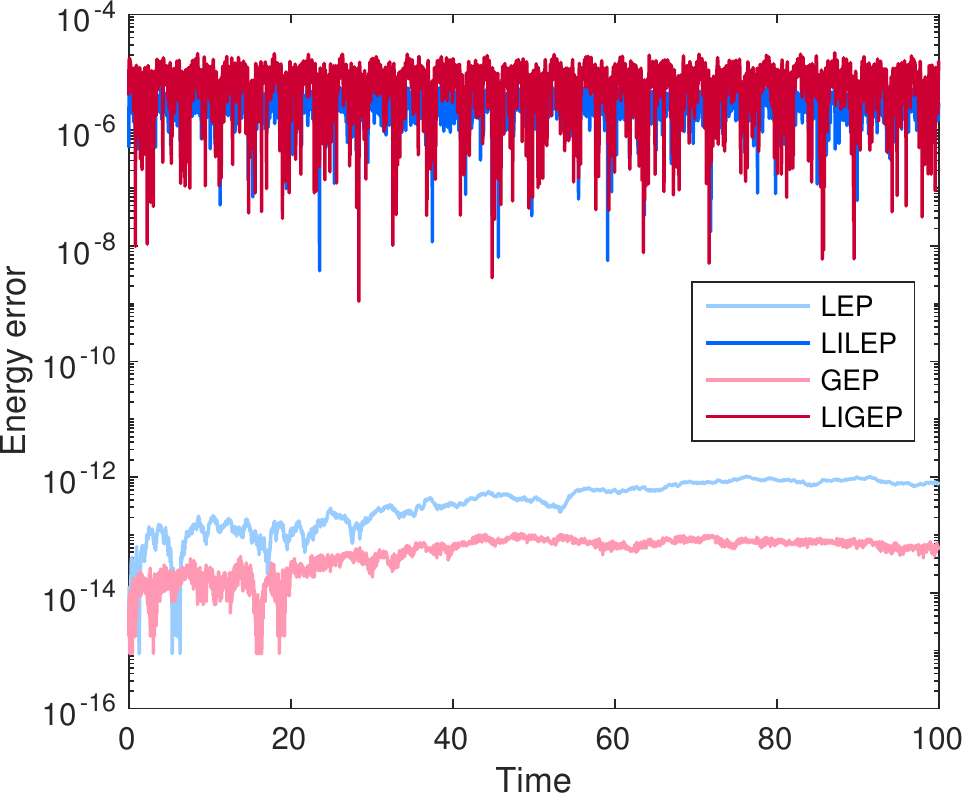}
        \end{minipage}      
      \caption{Error in discrete approximations to the global energy, by our methods and the fully implicit schemes of Gong et al. \textit{Left:} The error in \eqref{discrete global energy from local energy kdv} for LEP/LILEP and the error in \eqref{discrete global energy from global energy kdv} for GEP/LIGEP, for test problem 2 solved with $M=250$ discretization points in space and time step $\Delta t = 0.01$. \textit{Right:} The error in \eqref{discrete global energy from local energy kdv_LECpaper} for LEP/LILEP and the error in \eqref{discrete global energy from global energy kdv_LECpaper} for GEP/LIGEP.}
\label{KdV_soliton_energies}
\end{figure}

\begin{figure}[ht]
		\centering
	  \begin{minipage}{.47\textwidth}
		\centering
                \includegraphics[width=0.99\linewidth]{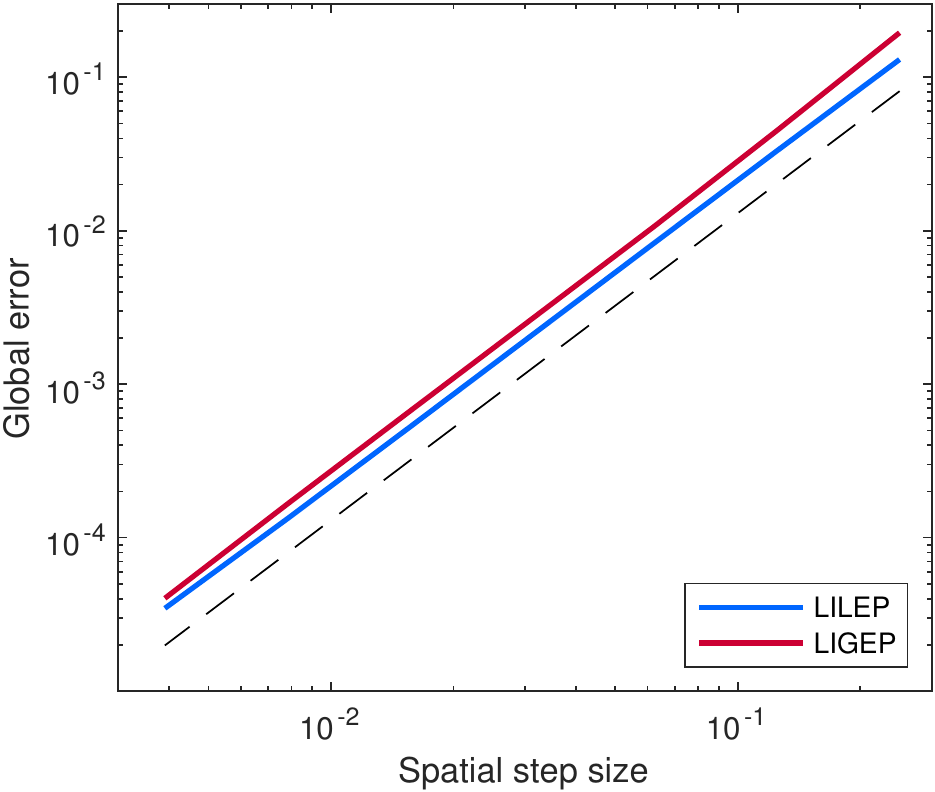}
        \end{minipage}\hspace{14pt}
          \begin{minipage}{.47\textwidth}
        \centering
                \includegraphics[width=0.99\linewidth]{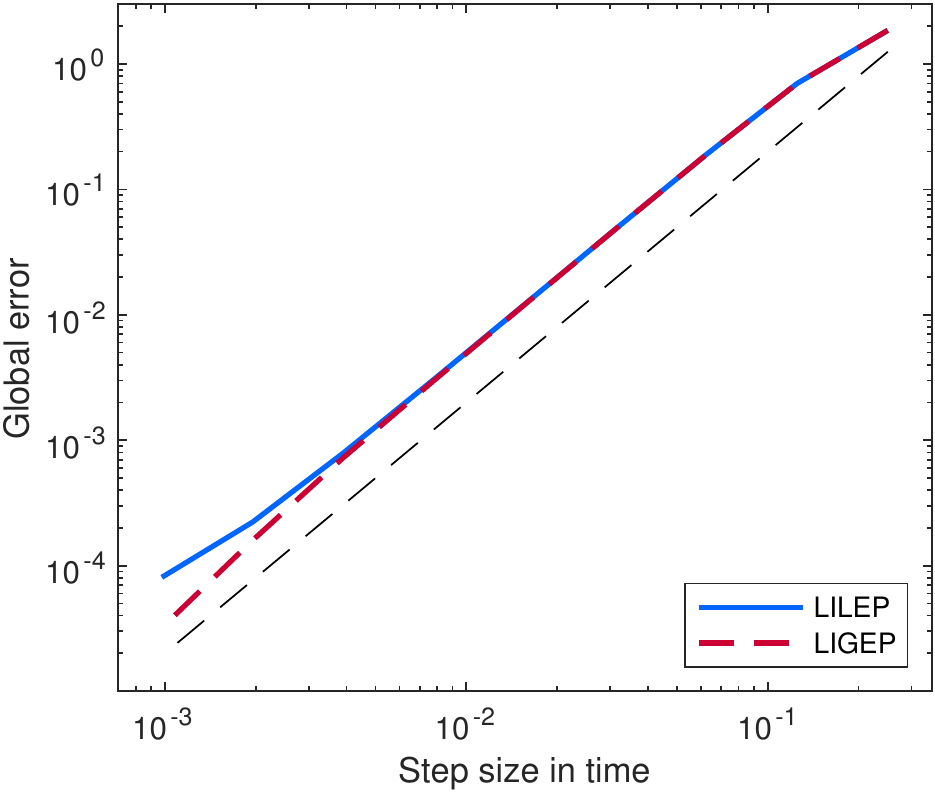}
        \end{minipage}  
      \caption{Order plots for the LILEP and LIGEP schemes, solving the second test problem for the KdV equation at time $t=1$. The black, dashed line is a reference line with slope $2$ in both plots. \textit{Left:} Fixed temporal step $\Delta t = 2 \times 10^{-4}$. \textit{Right:} Fixed spatial step $\Delta x = 4 \times 10^{-3}$.}\label{KdV_orderplots}
\end{figure}

\subsection{Zakharov--Kuznetsov equation}

Kahan's method is previously shown to have nice properties when applied to integrable ODE systems \cite{celledoni2012geometric,celledoni2012integrability}, and to perform well compared to other linearly implicit methods when applied to the KdV and Camassa--Holm equations \cite{eidnes2019linearly}, which are completely integrable PDEs. 
We wish to test our methods also on non-integrable systems, as well as on higher-dimensional problems. Therefore we consider the Zakharov--Kuznetsov equation, which is a non-integrable PDE \cite{Hu2008new,Nishiyama2012conservative}. 
This two-dimensional generalisation of the KdV equation has a variety of applications, see e.g.\ \cite{iwasaki1990cylindrical} for a brief summary.

Applying the LILEP method \eqref{LILECL_Scheme_2d} to the Zakharov--Kuznetsov equation \eqref{eq:Zakharov-Kuznetsov} multi-symplectified as described in Example \ref{ex:ZK}, we find
\begin{equation*}
\begin{split}
\delta_x \mu_t \mu_y \phi_{j,k}^n &= \mu_t \mu_x \mu_y u_{j,k}^n,\\
\frac{1}{2} \delta_t \mu_x \mu_y \phi_{j,k}^n + \delta_x \mu_t \mu_y v_{j,k}^n + \delta_y \mu_t \mu_x w_{j,k}^n &= \mu_t \mu_x \mu_y p_{j,k}^n - \frac{1}{2} \mu_x \mu_y u_{j,k}^n \mu_x \mu_y u_{j,k}^{n+1},\\
\delta_x \mu_t \mu_y w_{j,k}^n - \delta_y \mu_t \mu_x v_{j,k}^n &= 0,\\
-\frac{1}{2} \delta_t \mu_x \mu_y u_{j,k}^n - \delta_x \mu_t \mu_y p_{j,k}^n &= 0,\\
- \delta_x \mu_t \mu_y u_{j,k}^n + \delta_y \mu_t \mu_x q_{j,k}^n &= -\mu_t \mu_x \mu_y v_{j,k}^n,\\
- \delta_x \mu_t \mu_y q_{j,k}^n - \delta_y \mu_t \mu_x u_{j,k}^n &= -\mu_t \mu_x \mu_y w_{j,k}^n.
\end{split}
\end{equation*}
Upon eliminating all variables except $u$, we are left with
\begin{equation*}
\delta_t \mu_t \mu_x^3 \mu_y u_{j,k}^n + \frac{1}{2} \delta_x \mu_t \mu_x \mu_y (\mu_x \mu_y u_{j,k}^n \mu_x \mu_y u_{j,k}^{n+1}) + \delta_x^3 \mu_t^2 \mu_y^2 u_{j,k}^n + \delta_x \delta_y^2 \mu_t^2 \mu_x^2 u_{j,k}^n=0.
\end{equation*}
The operator $\mu_t$ is again superfluous. Hence we get the scheme
\begin{equation*}
\delta_t \mu_x^3 \mu_y u_{j,k}^n + \frac{1}{2} \delta_x \mu_x \mu_y (\mu_x \mu_y u_{j,k}^n \mu_x \mu_y u_{j,k}^{n+1}) + \delta_x^3 \mu_t \mu_y^2 u_{j,k}^n + \delta_x \delta_y^2 \mu_t \mu_x^2 u_{j,k}^n=0.
\end{equation*}
This scheme preserves 
\begin{equation*}
\begin{split}
\bar{\mathcal{E}}_L^n = \frac{1}{6} \Delta x \Delta y \sum_{j=0}^{M_x-1} \sum_{k=0}^{M_y-1} &\Big( 2 \delta_x \mu_y u_{j,k}^{n+1} \delta_x \mu_y u^n_{j,k} + (\delta_x \mu_y u_{j,k}^{n})^2 + 2 \delta_y \mu_x u_{j,k}^{n+1} \delta_y \mu_x u^n_{j,k} \\
& \, + (\delta_y \mu_x u_{j,k}^{n})^2 - (\mu_x \mu_y u_{j,k}^n)^2 (\mu_x \mu_y u_{j,k}^{n+1}) \Big),
\end{split}
\end{equation*}
which is a two-step discrete approximation of the energy
\begin{equation*}
\mathcal{E} = \int (\frac{1}{2}(\nabla u)^2 - \frac{1}{6}u^3) \, d\Omega.
\end{equation*}

Similarly, applying the linearly implicit global energy-preserving method \eqref{LIGECL_Scheme_2d} to \eqref{eq:ZKsystem}, we get the scheme
\begin{equation*}
\delta_t u_{j,k}^n + \frac{1}{2} (D_x(u^n u^{n+1}))_{j,k} + \mu_t(D_x^3(u^n))_{j,k} + \mu_t(D_xD_y^2(u^n))_{j,k} =0,
\end{equation*}
which preserves the two-step discrete energy approximation
\begin{equation*}
\begin{split}
\bar{\mathcal{E}}^n =& \frac{1}{6} \Delta x \, \Delta y \sum_{j=0}^{M_x-1} \sum_{k=0}^{M_y-1} \Big( 2 (D_xu^n)_{j,k} (D_xu^{n+1})_{j,k} + ((D_xu^n)_{j,k})^2 \\
&+ 2 (D_yu^n)_{j,k} (D_yu^{n+1})_{j,k} + ((D_yu^n)_{j,k})^2 - (u_{j,k}^n)^2 u_{j,k}^{n+1} \Big).
\end{split}
\end{equation*}

\subsubsection*{Test problem}

Taking a note from a numerical experiment performed in \cite{bridges2001multi}, we study the formation of cylindrical soliton pulses on the domain $\left[0,P\right] \times \left[0,P\right]$, $P=30$, following the initial condition
\begin{equation*}
u_0(x,y) = 3 c \, \mathrm{sech}^2 \big(\frac{1}{2}\sqrt{c}(x-P/2)\big) + \xi(y),
\end{equation*}
where $\xi(y)$ is a random perturbation.

Upon trying the different schemes we can immediately conclude that the local energy-preserving schemes are superior for this problem when compared to the global energy-preserving schemes. The GEP and LIGEP schemes are too costly when the pseudospectral operator is used, and gives oscillatory behaviour in the $y$-direction when the central difference operator is used, unless the discretization in this direction is very fine. Although the global energy-preserving schemes with the central difference operator are slightly faster then the local energy-preserving schemes, as can be seen in Table \ref{tab:timecost_ZK}, this is undermined by the cost of the extra discretization points needed to avoid oscillations in the former case. As was the case for the KdV problem, we see little difference between the linearly implicit schemes and their fully implicit counterparts. This can be seen in Figure \ref{fig:ZK_allplots}, as can the oscillations in $y$-direction of the solution found by the GEP and LIGEP methods. The plots in Figure \ref{fig:ZK_allplots} can be compared to the plot in Figure \ref{fig:ZK_big}, where the same problem is solved by the LILEP method using finer discretization in space and time. The initial random perturbation in $y$-direction over $75$ points is then transferred over to $225$ points using linear interpolation.

\begin{figure}[ht]
		\centering
	  \begin{minipage}{.47\textwidth}
		\centering
                \includegraphics[width=0.99\linewidth]{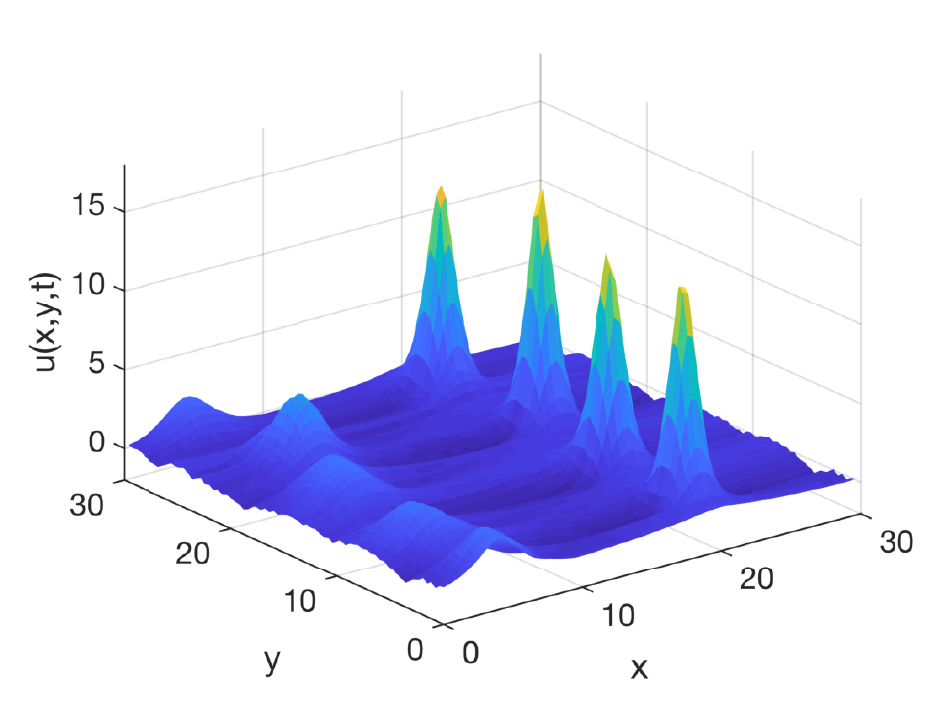}
        \end{minipage}\hspace{14pt}
          \begin{minipage}{.47\textwidth}
        \centering
                \includegraphics[width=0.99\linewidth]{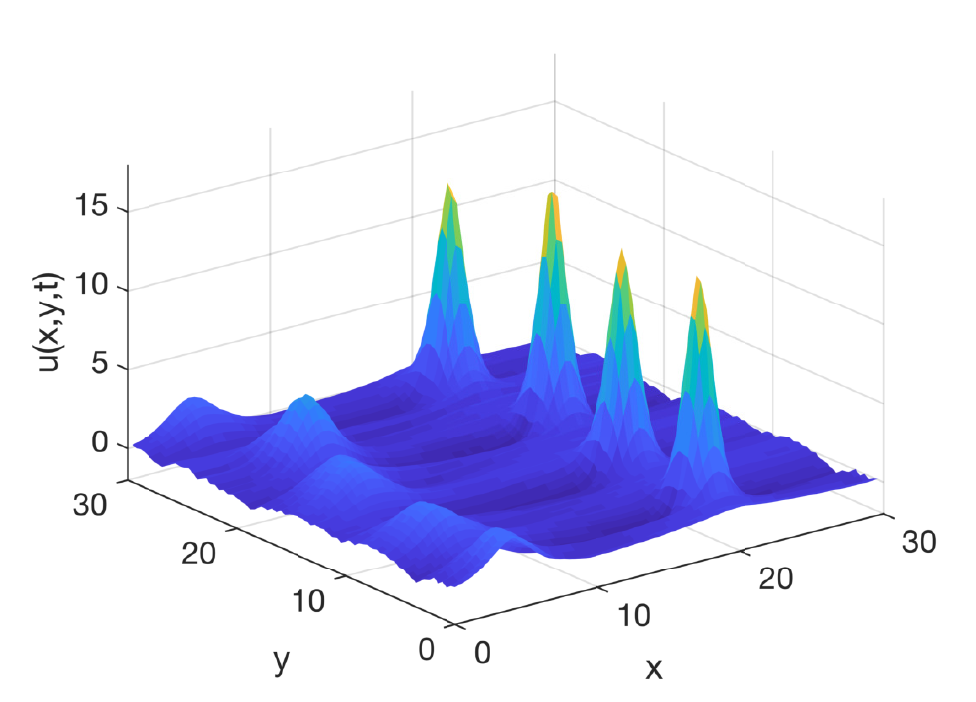}
        \end{minipage}  
        	  \begin{minipage}{.47\textwidth}
		\centering
                \includegraphics[width=0.99\linewidth]{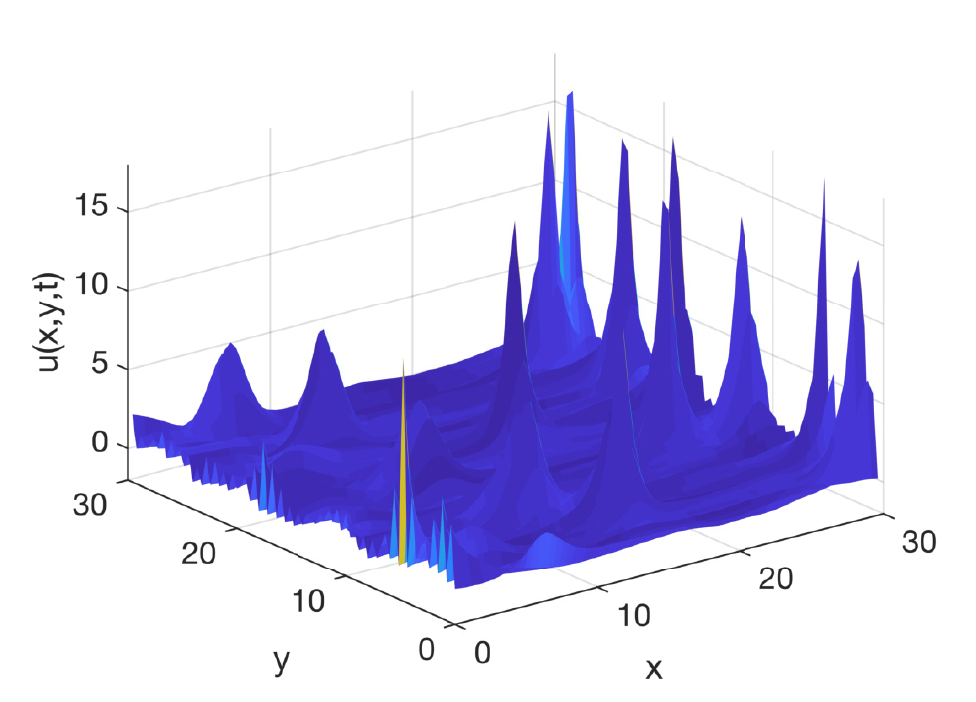}
        \end{minipage}\hspace{14pt}
          \begin{minipage}{.47\textwidth}
        \centering
                \includegraphics[width=0.99\linewidth]{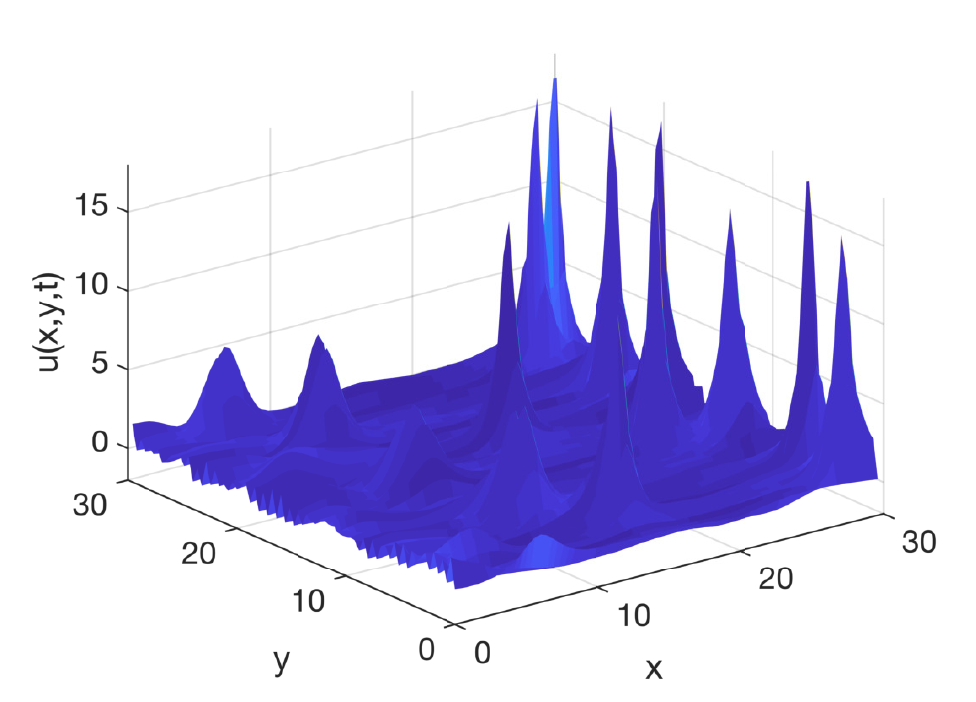}
        \end{minipage}  
      \caption{The test problem of the Zakharov--Kuznetsov equation solved at time $t=15$ by the different schemes, with $M=M_x=M_y=75$ points in each spatial direction and $\Delta t = 0.1$.}
      \label{fig:ZK_allplots}
\end{figure}

\begin{figure}[ht]
\centering
      \includegraphics[width=0.78\textwidth]{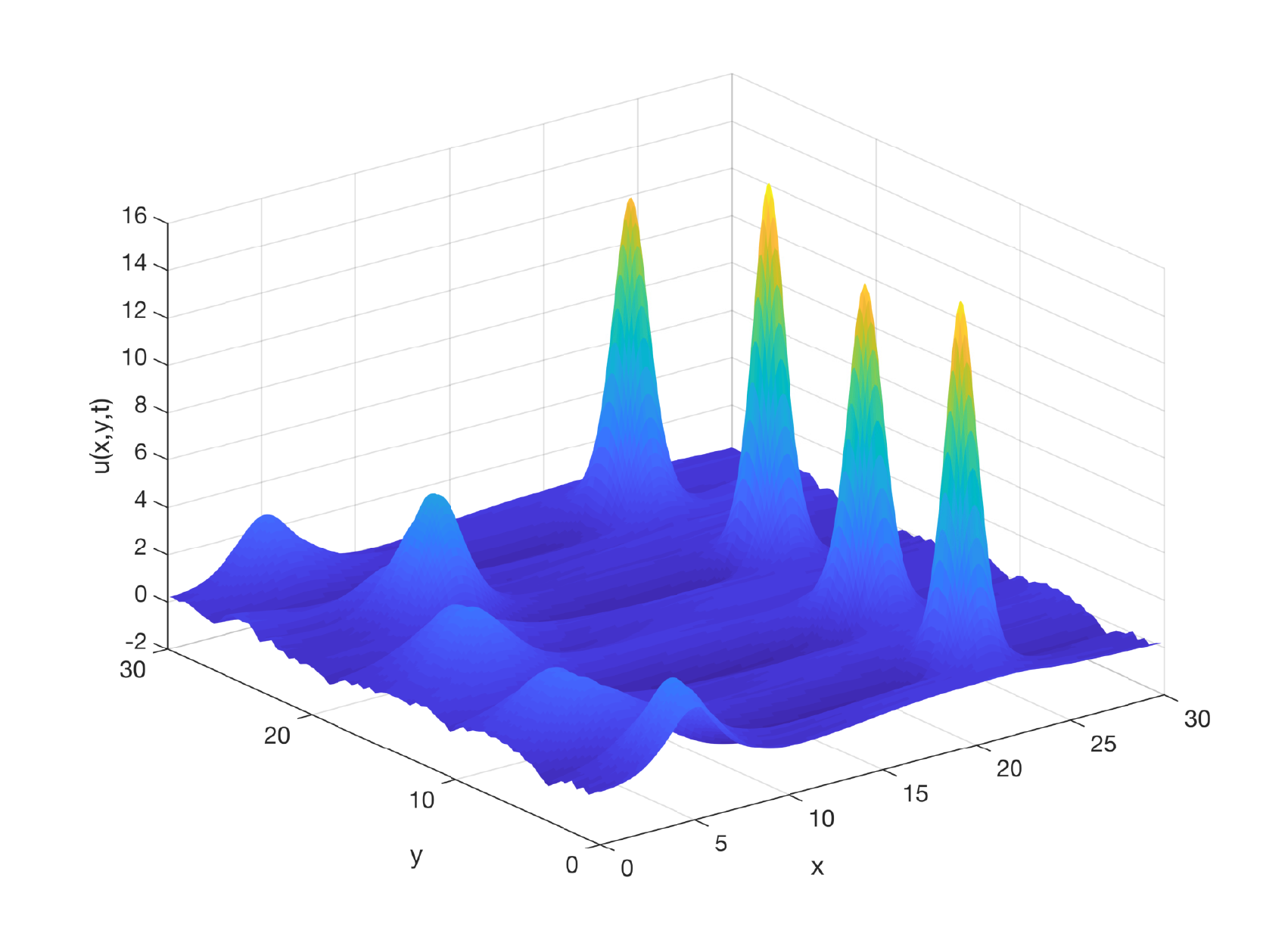}
      \caption{The test problem of the Zakharov--Kuznetsov equation solved at time $t=15$ by the LILEP scheme, with $M=M_x=M_y=225$ discretization points in each spatial direction and a temporal step size $\Delta t = 0.001$.}\label{fig:ZK_big}
\end{figure}

\begin{table}[ht]
{
\caption{Running time, in seconds, for computing $100$ steps in time by the various schemes and various number of discretization points $M=M_x=M_y$ in each spatial direction, solving our test problem for the Zakharov--Kuznetsov equation.}
\label{tab:timecost_ZK}
\begin{center}
\begin{tabular}{lcccccccc}
\hline
\multicolumn{1}{c}{$M$} & $45$ & $75$ & $105$ & $135$ & $165$ & $195$ & $225$ & $255$\\
\hline
LEP      	& 5.10 & 32.20 & 48.43 & 101.59 & 125.23 & 258.64 & 353.98 & 510.00     \\
LILEP    	& 2.04 & 8.87 & 14.57 & 31.02     & 37.25  & 78.98 & 108.02 & 157.91  \\
GEP 			& 3.62 & 19.54 & 41.87 & 73.59   & 122.31 & 186.74 & 258.19 & 352.36     \\
LIGEP 		& 1.38 & 6.00 & 13.45 & 23.79     & 39.31  & 60.27  & 83.32  & 113.13   \\
\hline
\end{tabular}
\end{center}
}
\end{table}

\section{Concluding remarks}\label{concluding remarks}
In this paper, we propose two types of linearly implicit methods with conservation properties for cubic invariants of multi-symplectic PDEs. The linearly implicit local energy-preserving (LILEP) method preserves a discrete approximation to the local energy conservation law, and by extension, the global energy whenever periodic boundary conditions are considered. The linearly implicit global energy-preserving (LIGEP) method preserves the global energy without inheriting the local preservation from the continuous system. 

We test our methods on two PDEs: the one-dimensional, integrable Korteweg--de Vries (KdV) equation and the two-dimensional, non-integrable Zakharov--Kuznetsov equation. The numerical experiments confirm that the proposed methods are of second order both in space and time and that they preserve the expected local and global energy conservation laws. We have observed excellent stability properties for the LILEP scheme in particular, and very high accuracy in the LIGEP scheme even for quite coarse discretization when a Fourier pseudospectral operator is used to approximate the spatial derivative. Compared to the fully implicit methods of Gong et al.\ in \cite{gong2014some}, which was an inspiration for this paper, our methods show comparable wave profiles, global errors and energy errors, at a significantly lower computational cost. For two-dimensional problems, where fully implicit schemes quickly become very expensive to compute, the combination of local energy-preservation and a linearly implicit method seems to provide for a very competitive method.

Although we have only considered the preservation of cubic invariants in this paper, our schemes can be extended to preserve higher order polynomials by the polarisation techniques for generalising Kahan's method suggested in \cite{celledoni2015discretization}. This would result in $(p-2)$-step methods for preservation of a discrete $p$-order polynomial invariant. Although the idea behind this is clear, obtaining the concrete schemes is not straightforward, and thus we leave it for future research.

\subsection*{Acknowledgements}
This work was supported by the European Union's Horizon 2020 research and innovation programme under the Marie Sk{\l{}}odowska-Curie grant agreement No. 691070. The authors wish to express gratitude to Elena Celledoni and Brynjulf Owren for constructive discussions and helpful suggestions during our work on this paper, and to Benjamin Tapley for helping with the language.

\bibliography{linimplocbibRev}

\begin{thebibliography}{10}

\bibitem{hairer2006geometric}
E.~Hairer, C.~Lubich, and G.~Wanner, {\em Geometric numerical integration},
  vol.~31 of {\em Springer Series in Computational Mathematics}.
\newblock Springer-Verlag, Berlin, second~ed., 2006.
\newblock Structure-preserving algorithms for ordinary differential equations.

\bibitem{furihata2011discrete}
D.~Furihata and T.~Matsuo, {\em Discrete variational derivative method}.
\newblock Chapman \& Hall/CRC Numerical Analysis and Scientific Computing, CRC
  Press, Boca Raton, FL, 2011.
\newblock A structure-preserving numerical method for partial differential
  equations.

\bibitem{christiansen2011topics}
S.~H. Christiansen, H.~Z. Munthe-Kaas, and B.~Owren, ``Topics in
  structure-preserving discretization,'' {\em Acta Numer.}, vol.~20,
  pp.~1--119, 2011.

\bibitem{Feynmanbook}
R.~P. Feynman, R.~B. Leighton, and M.~Sands, {\em The {F}eynman lectures on
  physics. {V}ol. 1: {M}ainly mechanics, radiation, and heat}.
\newblock Addison-Wesley Publishing Co., Inc., Reading, Mass.-London, 1963.

\bibitem{bridges1997multi}
T.~J. Bridges, ``Multi-symplectic structures and wave propagation,'' vol.~121,
  pp.~147--190, 1997.

\bibitem{li1995finite}
S.~Li and L.~Vu-Quoc, ``Finite difference calculus invariant structure of a
  class of algorithms for the nonlinear {K}lein-{G}ordon equation,'' {\em SIAM
  J. Numer. Anal.}, vol.~32, no.~6, pp.~1839--1875, 1995.

\bibitem{labudde1975energy}
R.~A. LaBudde and D.~Greenspan, ``Energy and momentum conserving methods of
  arbitrary order for the numerical integration of equations of motion. {II}.
  {M}otion of a system of particles,'' {\em Numer. Math.}, vol.~26, no.~1,
  pp.~1--16, 1976.

\bibitem{mclachlan1999geometric}
R.~I. McLachlan, G.~R.~W. Quispel, and N.~Robidoux, ``Geometric integration
  using discrete gradients,'' {\em R. Soc. Lond. Philos. Trans. Ser. A Math.
  Phys. Eng. Sci.}, vol.~357, no.~1754, pp.~1021--1045, 1999.

\bibitem{brugnano10hbv}
L.~Brugnano, F.~Iavernaro, and D.~Trigiante, ``Hamiltonian boundary value
  methods (energy preserving discrete line integral methods),'' {\em JNAIAM. J.
  Numer. Anal. Ind. Appl. Math.}, vol.~5, no.~1-2, pp.~17--37, 2010.

\bibitem{celledoni2012preserving}
E.~Celledoni, V.~Grimm, R.~I. McLachlan, D.~I. McLaren, D.~O'Neale, B.~Owren,
  and G.~R.~W. Quispel, ``Preserving energy resp. dissipation in numerical
  {PDE}s using the ``average vector field'' method,'' {\em J. Comput. Phys.},
  vol.~231, no.~20, pp.~6770--6789, 2012.

\bibitem{bridges2001multi}
T.~J. Bridges and S.~Reich, ``Multi-symplectic spectral discretizations for the
  {Z}akharov--{K}uznetsov and shallow water equations,'' {\em Phys. D},
  vol.~152/153, pp.~491--504, 2001.
\newblock Advances in nonlinear mathematics and science.

\bibitem{bridges1997geometric}
T.~J. Bridges, ``A geometric formulation of the conservation of wave action and
  its implications for signature and the classification of instabilities,''
  {\em Proc. Roy. Soc. London Ser. A}, vol.~453, no.~1962, pp.~1365--1395,
  1997.

\bibitem{leimkuhler2004simulating}
B.~Leimkuhler and S.~Reich, {\em Simulating {H}amiltonian dynamics}, vol.~14 of
  {\em Cambridge Monographs on Applied and Computational Mathematics}.
\newblock Cambridge University Press, Cambridge, 2004.

\bibitem{Sunmultisym}
Y.~Sun and P.~S.~P. Tse, ``Symplectic and multisymplectic numerical methods for
  {M}axwell's equations,'' {\em J. Comput. Phys.}, vol.~230, no.~5,
  pp.~2076--2094, 2011.

\bibitem{EPmultisym}
Y.-W. Li and X.~Wu, ``General local energy-preserving integrators for solving
  multi-symplectic {H}amiltonian {PDE}s,'' {\em J. Comput. Phys.}, vol.~301,
  pp.~141--166, 2015.

\bibitem{Hydonlocalconserv}
G.~Frasca-Caccia and P.~E. Hydon, ``Locally conservative finite difference
  schemes for the modified {KDV} equation,'' {\em J. Comput. Dyn.}, vol.~6,
  no.~2, pp.~307--323, 2019.

\bibitem{MR2221062}
P.~Chartier, E.~Faou, and A.~Murua, ``An algebraic approach to invariant
  preserving integrators: the case of quadratic and {H}amiltonian invariants,''
  {\em Numer. Math.}, vol.~103, no.~4, pp.~575--590, 2006.

\bibitem{wang2008local}
Y.~Wang, B.~Wang, and M.~Qin, ``Local structure-preserving algorithms for
  partial differential equations,'' {\em Sci. China Ser. A}, vol.~51, no.~11,
  pp.~2115--2136, 2008.

\bibitem{marsden1998multisymplectic}
J.~E. Marsden, G.~W. Patrick, and S.~Shkoller, ``Multisymplectic geometry,
  variational integrators, and nonlinear {PDE}s,'' {\em Comm. Math. Phys.},
  vol.~199, no.~2, pp.~351--395, 1998.

\bibitem{reich2000multi}
S.~Reich, ``Multi-symplectic {R}unge-{K}utta collocation methods for
  {H}amiltonian wave equations,'' {\em J. Comput. Phys.}, vol.~157, no.~2,
  pp.~473--499, 2000.

\bibitem{gong2014some}
Y.~Gong, J.~Cai, and Y.~Wang, ``Some new structure-preserving algorithms for
  general multi-symplectic formulations of {H}amiltonian {PDE}s,'' {\em J.
  Comput. Phys.}, vol.~279, pp.~80--102, 2014.

\bibitem{li2015general}
Y.-W. Li and X.~Wu, ``General local energy-preserving integrators for solving
  multi-symplectic {H}amiltonian {PDE}s,'' {\em J. Comput. Phys.}, vol.~301,
  pp.~141--166, 2015.

\bibitem{kahan1993unconventional}
W.~Kahan, ``Unconventional numerical methods for trajectory calculations,''
  {\em Unpublished lecture notes}, vol.~1, p.~13, 1993.

\bibitem{celledoni2012geometric}
E.~Celledoni, R.~I. McLachlan, B.~Owren, and G.~R.~W. Quispel, ``Geometric
  properties of {K}ahan's method,'' {\em J. Phys. A}, vol.~46, no.~2,
  pp.~025201, 12, 2013.

\bibitem{celledoni2012integrability}
E.~Celledoni, R.~I. McLachlan, D.~I. McLaren, B.~Owren, and G.~R.~W. Quispel,
  ``Integrability properties of {K}ahan's method,'' {\em J. Phys. A}, vol.~47,
  no.~36, pp.~365202, 20, 2014.

\bibitem{celledoni2018geometric}
E.~Celledoni, D.~I. McLaren, B.~Owren, and G.~R.~W. Quispel, ``Geometric and
  integrability properties of {K}ahan's method: the preservation of certain
  quadratic integrals,'' {\em J. Phys. A}, vol.~52, no.~6, pp.~065201, 9, 2019.

\bibitem{matsuo2001dissipative}
T.~Matsuo and D.~Furihata, ``Dissipative or conservative finite-difference
  schemes for complex-valued nonlinear partial differential equations,'' {\em
  J. Comput. Phys.}, vol.~171, no.~2, pp.~425--447, 2001.

\bibitem{dahlby2011general}
M.~Dahlby and B.~Owren, ``A general framework for deriving integral preserving
  numerical methods for {PDE}s,'' {\em SIAM J. Sci. Comput.}, vol.~33, no.~5,
  pp.~2318--2340, 2011.

\bibitem{eidnes2019linearly}
S.~Eidnes, L.~Li, and S.~Sato, ``Linearly implicit structure-preserving schemes
  for {H}amiltonian systems,'' {\em arXiv preprint, arXiv:1901.03573}, 2019.

\bibitem{cai2018partitioned}
W.~Cai, H.~Li, and Y.~Wang, ``Partitioned averaged vector field methods,'' {\em
  J. Comput. Phys.}, vol.~370, pp.~25--42, 2018.

\bibitem{yang2017numerical}
X.~Yang, J.~Zhao, and Q.~Wang, ``Numerical approximations for the molecular
  beam epitaxial growth model based on the invariant energy quadratization
  method,'' {\em J. Comput. Phys.}, vol.~333, pp.~104--127, 2017.

\bibitem{jiang2019linearly}
C.~Jiang, Y.~Gong, W.~Cai, and Y.~Wang, ``A linearly implicit
  structure-preserving scheme for the {C}amassa--{H}olm equation based on
  multiple scalar auxiliary variables approach,'' {\em arXiv preprint,
  arXiv:1907.00167}, 2019.

\bibitem{jiang2018linear}
C.~Jiang, W.~Cai, and Y.~Wang, ``A linear-implicit and local energy-preserving
  scheme for the sine-{G}ordon equation based on the invariant energy
  quadratization approach,'' {\em arXiv preprint, arXiv:1808.06854}, 2018.

\bibitem{li13new}
H.~Li and J.~Sun, ``A new multi-symplectic {E}uler box scheme for the {BBM}
  equation,'' {\em Math. Comput. Modelling}, vol.~58, no.~7-8, pp.~1489--1501,
  2013.

\bibitem{Boussinesq}
A.~Dur\'{a}n, D.~Dutykh, and D.~Mitsotakis, ``On the multi-symplectic structure
  of {B}oussinesq-type systems. {I}: {D}erivation and mathematical
  properties,'' {\em Phys. D}, vol.~388, pp.~10--21, 2019.

\bibitem{Cohnmultisym}
D.~Cohen, T.~Matsuo, and X.~Raynaud, ``A multi-symplectic numerical integrator
  for the two-component {C}amassa-{H}olm equation,'' {\em J. Nonlinear Math.
  Phys.}, vol.~21, no.~3, pp.~442--453, 2014.

\bibitem{celledoni2015discretization}
E.~Celledoni, R.~I. McLachlan, D.~I. McLaren, B.~Owren, and G.~R.~W. Quispel,
  ``Discretization of polynomial vector fields by polarization,'' {\em Proc.
  A.}, vol.~471, no.~2184, pp.~20150390, 10, 2015.

\bibitem{moore2003multi}
B.~E. Moore and S.~Reich, ``Multi-symplectic integration methods for
  {H}amiltonian {PDE}s,'' {\em Future Generation Computer Systems}, vol.~19,
  no.~3, pp.~395--402, 2003.

\bibitem{zakharov1974threedimensional}
V.~Zakharov and E.~Kuznetsov, ``Three-dimensional solitons,'' {\em Zh. Eksp.
  Teor. Fiz}, vol.~66, pp.~594--597, 1974.

\bibitem{chen2011multi}
Y.~Chen, S.~Song, and H.~Zhu, ``The multi-symplectic {F}ourier pseudospectral
  method for solving two-dimensional {H}amiltonian {PDE}s,'' {\em J. Comput.
  Appl. Math.}, vol.~236, no.~6, pp.~1354--1369, 2011.

\bibitem{zabusky1965interaction}
N.~J. Zabusky and M.~D. Kruskal, ``Interaction of "solitons" in a collisionless
  plasma and the recurrence of initial states,'' {\em Phys. Rev. Lett.},
  vol.~15, no.~6, p.~240, 1965.

\bibitem{zhao2000multisymplectic}
P.~F. Zhao and M.~Z. Qin, ``Multisymplectic geometry and multisymplectic
  {P}reissmann scheme for the {K}d{V} equation,'' {\em J. Phys. A}, vol.~33,
  no.~18, pp.~3613--3626, 2000.

\bibitem{ascher2005symplectic}
U.~M. Ascher and R.~I. McLachlan, ``On symplectic and multisymplectic schemes
  for the {K}d{V} equation,'' {\em J. Sci. Comput.}, vol.~25, no.~1-2,
  pp.~83--104, 2005.

\bibitem{Hu2008new}
H.-C. Hu, ``New exact solutions of {Z}akharov--{K}uznetsov equation,'' {\em
  Commun. Theor. Phys. (Beijing)}, vol.~49, no.~3, pp.~559--561, 2008.

\bibitem{Nishiyama2012conservative}
H.~Nishiyama, T.~Noi, and S.~Oharu, ``Conservative finite difference schemes
  for the generalized {Z}akharov--{K}uznetsov equations,'' {\em J. Comput.
  Appl. Math.}, vol.~236, no.~12, pp.~2998--3006, 2012.

\bibitem{iwasaki1990cylindrical}
H.~Iwasaki, S.~Toh, and T.~Kawahara, ``Cylindrical quasi-solitons of the
  {Z}akharov--{K}uznetsov equation,'' {\em Phys. D}, vol.~43, no.~2-3,
  pp.~293--303, 1990.

\end{thebibliography}
\bibliographystyle{ieeetr}
\end{document}